\documentclass[12pt,letterpaper]{amsart}
\usepackage{ amsmath,amsthm, amsbsy,amsfonts,amssymb, txfonts, shuffle}
\usepackage[normalem]{ulem}
\usepackage{stmaryrd,mathrsfs,graphicx, supertabular, amscd, tikz-cd, tikz}
\tikzcdset{row sep/my size/.initial=6mm}
\usepackage{stackrel}
\usetikzlibrary{matrix,arrows,decorations.pathmorphing}

\usepackage{ulem}

\usepackage[active]{srcltx}
\allowdisplaybreaks
\numberwithin{equation}{section}

\usepackage[
	hypertexnames=false,
	hyperindex,
	pagebackref,
	breaklinks=true,
	bookmarks=false,
	colorlinks,
	linkcolor=blue,
	citecolor=red,
	urlcolor=blue,
]{hyperref}
\usepackage{hyperref}
\def\CC{{\mathbb C}}

\def\HH{{\mathbb H}}

\def\Kbold{{\mathbf K}}

\def\Nbold{{\mathbf N}} 
\def\PP{{\mathbb P}}

\def\QQ{{\mathbb Q}} 
\def\RR{{\mathbb R}}

\def\ZZ{{\mathbb Z}}

\def\Ibold{\mathbf {I}}
\def\Jbold{\mathbf {J}}

\def\ssm{\smallsetminus}

\def\g{\gamma}
\def\eps{\varepsilon}

\def\tri{\triangle}

\newcommand{\p}{\partial}

\def\Bcal{{\mathcal B}}
\def\Ccal{{\mathcal C}}

\def\Hcal{{\mathcal H}} 
\def\Ical{{\mathcal I}} 
 
\def\Kcal{{\mathcal K}}

\def\Scal{{\mathcal S}}

\def\uHcal{{\underline\Hcal}}

\newcommand\FB{\underline{\mathcal {F\!B}}}

\newcommand \Zmod{\underline{\mathcal A b}}

\def\Cscr{{\mathscr C}}

\def\la{\langle}
\def\lla{\langle\!\langle} 
\def\ra{\rangle}
\def\rra{\rangle\!\rangle} 
\def\half{{\tfrac{1}{2}}}

\def\Sfrak{\mathfrak{S}}

\def\pt{{\scriptscriptstyle\bullet}}

\newcommand\ad{\operatorname{ad}}

\newcommand\att{\operatorname{att}}

\newcommand\aut{\operatorname{Aut}}

\newcommand\conf{\mathscr{C}\!\mathit{onf}\!}
\newcommand\cfg{\mathit{cfg}}

\newcommand\gr{\operatorname{gr}}

\newcommand\Hom{\operatorname{Hom}}

\newcommand\Htp{\operatorname{Htp}}
\newcommand\im{\operatorname{Im}}

\newcommand\Ker{\operatorname{Ker}}

\newcommand\Mod{\operatorname{Mod}}

\newcommand\re{\operatorname{Re}}

\newcommand\sign{\operatorname{sign}}

\newcommand\Tcal{\operatorname{{\mathcal T}}}

\newtheorem{theorem}{Theorem}[section]
\newtheorem{lemma}[theorem]{Lemma}
\newtheorem{proposition}[theorem]{Proposition}
\newtheorem{corollary}[theorem]{Corollary}

\theoremstyle{remark}
\newtheorem{example}[theorem]{Example}
\newtheorem{examples}[theorem]{Examples}

\newtheorem{remark}[theorem]{Remark}

\title[Configuration functor of a punctured space]{The configuration functor of a punctured space}
\author{Eduard Looijenga}
\author{Andreas Stavrou}
\email{e.j.n.looijenga@uu.nl, andreasstavrou@uchicago.edu}
\address{Mathematics Department, University of Chicago}

\begin{document}
\begin{abstract} 
Let $U$ be a space whose one point compactification $U^*$ is a CW-complex  for which the added point $*$ is the only $0$-cell. 
We observe that  the configuration space $\conf_n(U)$ of $n$ numbered distinct points in $U$ 
has no closed support homology in degree $<n$ and prove that  Borel-Moore homology group $H^{cl}_n(\conf_n(U))$ depends only on the fundamental group $\pi_1(U^*,*)$.  
We describe this homology group in terms of a presentation of $\pi_1(U^*,*)$.
 
A case of interest is when $U$ is a connected closed oriented surface  of positive genus minus a finite nonempty set. 
Then the  mapping class group $\Mod(U)$ of $U$  acts on both $\pi_1(U^*,*)$ and ${H^k}(\conf_n(U){)}\cong H^{cl}_
{2n-k}(\conf_n(U))$ and we prove that its action on the latter is
through its action on the nilpotent quotient $\pi_1(U^*,*)/ \pi_1(U^*,*)^{(k+1)}$.  Furthermore, we give an example of a mapping class of a once punctured closed surface $U$ which acts trivially on ${H^n}(\conf_n(U))$, but not on the nilpotent quotient $\pi_1(U^*,*)/ \pi_1(U^*,*)^{(n+1)}$. The former generalizes a theorem of Bianchi--Miller--Wilson and the latter  disproves a conjecture of theirs.
\end{abstract}
\maketitle
 
\section{Introduction}
Let $U$ be a space with the property that its one point compactification $U^*$ admits a  finite cell structure for which the added point $*$ is the only $0$-cell. So such a space cannot have a compact connected component. An example  to keep in mind is that of the  interior of a compact topological manifold with boundary without closed  components (so that $U^*$ is obtained by contracting the boundary); we then say that $U$ is an \emph{open manifold of finite type}.

This paper is about the configuration spaces of such $U$. We find it convenient to approach these as a functor.
To be concrete,  we consider for any   finite set $N$ of size $n\ge 0$  the space $\conf_N(U)$ of injective maps $N\to U$ (so this is a singleton when $N=\emptyset$). When  we  write $\conf_n(U)$, it is understood that $n$ takes the role of $[n]=\{1,2,\dots, n\}$. 

We regard $\conf_N(U)$ as a subset of  $U^*{}^N$. Its  complement, which we  shall denote by $D_N(U)$,  consists of the set of maps $f: N\to U^*$  which have  $*$  in their image  or fail to be injective (so  if  $N=\emptyset$, then  $D_N(U)=\emptyset$).  This subset  is obviously closed. Here is a  simple, but as we will see,  basic example.

\begin{example}\label{example:basic}
If $U=\RR$, then its $1$-point compactification is the real projective line $\PP^1(\RR)$. For any  injective map 
$f: N\to \RR$ from a finite set $N$,  the natural order of $\RR$ induces a total order on $N$. This total order, given by a bijection $\varphi: [n]\cong N$, characterizes the connected component of  $f$ in $\conf_N(\RR)$, for it will consist of the subset of  $\RR^N$ given by
$t_{\varphi(1)}<t_{\varphi(1)}<\cdots <t_{\varphi(n)}$. We shall think of this as the relative interior of the $n$-simplex $-\infty \le  s_1\le \cdots\le s_n\le \infty$ mapped to  $\PP^1(\RR)^N$ according to $\varphi$, so that  its boundary maps to $D_N(\RR)$. If $\Nbold$ stands for $N$ equipped with this total order, then we denote the $n$-simplex thus defined by $\tri^\Nbold$.
\end{example}

This example also illustrates the  following (rather easy) theorem, which will  be our point of departure.

\begin{theorem}\label{thm:connectedness}
The pair $(U^*{}^N, D_N(U))$ is $(n-1)$-connected. In fact, it has the structure of a relative finite cell  complex with cells of dimension $\ge n$ only,  such that 
the  $n$-cells of $\conf_N(U)$  are the connected components of $\conf_N$ of the union of the $1$-cells in  $U$ and, more generally,
every $(n+d)$-cell of $\conf_N(U)$ is contained in $\conf_N$ of the  $d$-skeleton of $U$.
\end{theorem}

So  the first possibly nonzero homology group of  this pair is in degree $n$; we denote that homology group 
\[
\Hcal_N(U):= H_n(U^*{}^N, D_N(U)).
\]
Observe that this is also the homology with closed support $H_n^{cl}(\conf_N(U))$ and that for $N$ empty resp.\ a singleton, this is $\ZZ$ resp.\ $H_1^{cl}(U)$.  

Example \ref{example:basic} shows that $\Hcal_N(\RR)$ is the free abelian group generated by the collection $\tri^\Nbold$,  which is indexed by the  set of total orders on $N$.
 \smallskip

The assignment $N\mapsto\Hcal_N(U)$  has  a number of naturality  properties worth spelling out.  First, there is the functoriality in $N$:
a bijection of finite sets $N\cong N'$ induces an isomorphism $\Hcal_N(U)\cong \Hcal_{N'}(U)$; in particular, the permutation group $\Sfrak_N$ of $N$ acts on $\Hcal_N(U)$.   So this defines a functor from the groupoid 
$\FB$ whose objects are finite sets and whose morphisms are bijections  ($F$ for finite, $B$ for bijections) to the category of abelian groups:
 \[
\uHcal(U):  \FB\to \Zmod.
\]
This functor is monoidal:  if $N$ is the disjoint union of $N'$ and $N''$, then 
the inclusion  $\conf_{N}(U)\subset \conf_{N'}(U)\times \conf_{N''}(U)$ defines a K\"unneth  product 
\begin{equation*}
\times: \Hcal_{N'}(U)\otimes \Hcal_{N''}(U)\to \Hcal_N(U)
\end{equation*}
 that   is associative and functorial in its arguments and  is the obvious isomorphism  if $N'$ or $N''$ is empty.

 The other type of functoriality is in the spatial argument:  for $U$ and $U'$ as above, a   continuous map $f:(U^*,*)\to (U'{}^*,*)$  of pointed spaces  induces a  map of pairs $(U^*{}^N, D_N(U))\to (U'{}^*{}^N, D_N(U'))$ and hence induces a map $\Hcal_N(U)\to \Hcal_N(U')$. This defines a natural transformation 
  \[
 \uHcal (f): \uHcal(U)\Rightarrow \uHcal(U')
  \]
of monoidal functors. It is clear that  this depends only on the relative  homotopy class of $f$. In particular the group 
 $\Htp(U^*,*)$ of self-homotopy equivalences of $(U^*,*)$ acts on $\Hcal_N(U)$, an  action which commutes with that of $\Sfrak_N$. 
So if we let $\underline{\Hcal\! pt}^*$  stand for the category  whose objects are the  spaces $U$ considered here and  morphisms are pointed homotopy classes as above, then we have defined 
a   functor 
\begin{equation}\label{eqn:2functor}
\uHcal :\underline{\Hcal\! pt}^*\to  \underline{\mathbf{Mon}}(\FB, \Zmod)
\end{equation}
 to the  functor category whose objects are monoidal functors and morphisms are their monoidal natural transformations. This functor   is itself   monoidal and is so in an even stronger sense:  if $U$ is the disjoint union of open subsets $U'$ and $U''$, then 
 $\conf_N(U)$ is the disjoint union of the products  $\conf_{N'}(U')\times \conf_{N\ssm N'}(U'')$, where $N'$ runs over all the subsets of $N$ and hence we have an identification
 \begin{equation}\label{eqn:spatial-lmonoidal}
 \Hcal_N(U)\cong \oplus_{N'\subset N} \Hcal_{N'}(U')\otimes \Hcal_{N\ssm N'}(U'').
\end{equation}
by  the K\"unneth formula. 
The right hand side  comes from an obvious monoidal structure on  $\underline{\mathbf{Mon}}(\FB, \Zmod)$, so that this amounts to  a strict equivalence 
$ \uHcal(U)\sim\uHcal(U')\otimes \uHcal(U'')$. 
Both functors  $\uHcal(U)$ and $\uHcal$  become \emph{symmetric} monoidal if  we replace  $\Zmod$ by the category  \emph{graded} abelian groups (which assigns to $\Hcal_N(U)$ the degree $|N|$), as there is a Koszul rule to obey when we exchange factors.

Since the  fundamental group $\pi_1(U^*,*)$ plays  a central role in what follows, we shall  abbreviate this group by $\pi_U$. 
Theorem \ref{thm:connectedness} implies that $\Hcal_N(U)$ depends only on the $2$-skeleton of $U^*$, but the homotopy discussion leads even to:

\begin{corollary}\label{cor:fundgrpdependence}
The group $\Hcal_N(U)$ depends only on the fundamental group $\pi_U=\pi_1(U^*,*)$;  in particular, $\Htp(U^*,*)$ acts on $\Hcal_N(U)$ through its action on $\pi_U$. To be precise, the monoidal  functor 
$\uHcal$  factors through  a monoidal functor $\underline{\pi\Hcal}$ from the  category 
of  finitely presented groups to $\underline{\mathbf{Mon}}(\FB, \Zmod)$ taking  a free product $G_1*G_2$ to $\underline{\pi\Hcal}(G_1)\otimes \underline{\pi\Hcal}(G_2)$,  such that an isomorphism $G\cong \pi_U$ induces an equivalence  
$\underline{\pi\Hcal}(G)\sim\uHcal(U)$.
\end{corollary}

If $U$ happens to be an oriented  open  $d$-manifold of finite type, then a total order on $N$ gives an orientation of $U^N$
and  then  
\[
\Hcal_N(U)\cong H^{(d-1)n}(\conf_N(U))\otimes (\sign_N)^{\otimes d}
\]
by Poincar\'e--Lefschetz duality, where  $\sign_N: \Sfrak_N\to \{\pm\}$ is the sign character. In this  setting, it is natural to restrict the $\Htp(U^*,*)$-action to that of the topological mapping class group $\Mod(U)$, i.e., the connected component group  of the group of orientation preserving self-homeomorphisms of $U$. Corollary \ref{cor:fundgrpdependence} above  shows that this action is through its action on $\pi_U$.
When $d=2$ and $U$ is connected (which in the end is the case that interests us most), then  $U$ is of the form $X\ssm \{p_0, \dots , p_l\}$, where $X$ is a closed oriented surface and the $p_i\in X$ are distinct. So $U^*$ has then the  homotopy type of a wedge of $X$ and $l$ circles. 
Since the  latter is the $1$-point compactification of the disjoint union of $X\ssm \{p_0\}$ and $l$ copies of $\RR$, it follows that 
\[
\uHcal(U)\cong \uHcal(X\ssm \{p_0\})\otimes \uHcal(\RR)^{\otimes l}.
\]
This gives the case $k=0$ of the following more general theorem.
\begin{theorem}\label{thm:manypuncturesalldegrees}
    For $0\le k\le n$ and a set $N$ of size $n$, the abelian group $H^{n-k}(\conf_N(U))$ is isomorphic to 
    $$
    \oplus_{(N_0, N_1,\dots, N_l )} H^{n_0-k}(\conf_{N_0}(X\ssm\{p_0\}))\otimes H^{n_1}(\conf_{N_1}(\RR))\otimes \dots H^{n_{l}}(\conf_{N_{l}}(\RR)),
    $$
    where the sum is over all partitions $N=N_0\sqcup N_1\sqcup \cdots\sqcup N_l $ of $N$ into numbered subsets and  $n_i=|N_i|$. For $k<0$, both $H^{n-k}(\conf_N(U))$ and the direct sum vanish.
\end{theorem}
In Section \ref{sec:manypunctures} we discuss to what extent this isomorphism is natural.
\smallskip

Here is a direct way to produce  elements of $\Hcal_N(U)$  from the fundamental group $\pi_U$.
Represent any $\g\in \pi_U$  by a map $\PP^1(\RR)\to U^*$ which takes $\infty$ to $*$ so that is defined 
$\Hcal_N(\g):  \Hcal_N(\RR)\to \Hcal_N(U)$.  Every total order on $N$ defines a generator $\tri^\Nbold$ of $\Hcal_N(\RR)$ and therefore gives  us an element $\Hcal_N(\g)(\tri^\Nbold)\in  \Hcal_N(U)$. We shall denote  the latter by $\tri^\Nbold(\g)$. It is clear that this extends to a linear map from the group ring $\ZZ\pi_U$ of $\pi_U$ to $\Hcal_N(U)$. Let us write  $\Lambda_{U}$ for  $\ZZ\pi_U$ and denote the kernel of the augmentation 
$\ZZ\pi_U\to \ZZ$ (which is constant $1$ on the basis of $\pi_U$) by $\Ical_U$. A theorem that goes back to Beilinson (see  e.g. \cite{looij:BDG}) implies  that  $\triangle^\Nbold(\g)$ depends only on the image of $\g$ in the truncated group ring $\Lambda_{U}|_n:=\Lambda_{U}/\Ical_{U}^{n+1}$, so that this construction factors through a linear map
\begin{gather*}
\triangle_U^\Nbold: \Lambda_{U}|_n\to  \Hcal_N(U).
\end{gather*}
This map is clearly $\Htp(U^*,*)$-equivariant. The obvious map $\pi_U\to  \Lambda_{U}|_n$ takes values in the group of units of $\Lambda_{U}|_n$ and is then a group homomorphism. Since the action of $\pi_U$ on  $\Lambda_{U}|_n$ by left multiplication is unipotent of degree $n$ (it acts trivially on each successive quotient
$\Ical_U^k/\Ical_U^{k+1}$),  the kernel  of this map contains the $(n+1)$st term 
$\pi_U^{(n+1)}$ of the central descending series, so that it factors through $\pi_U|_n:=\pi_U/\pi_U^{(n+1)}$.  
Let us also observe that since $1\in \Lambda_U$ represents the constant loop,  $\triangle_U^\Nbold$ takes on $1$ the value  $0$ when  $n>0$. So then  the image of $\triangle_U^\Nbold$ does not change if we restrict it  to $\Ical_{U}|_n:=\Ical_U/\Ical_U^{n+1}$. 

With the help of  the monoidal structure we now get many more elements: for any  partition  $N=N_1\sqcup \cdots\sqcup N_r$ into nonempty subsets (so the size $n_i$ of $N_i$ is positive) and a choice of a total order on each $N_i$ (indicated by the use of the boldface font $\Nbold_i$),  we get a linear map 
\begin{multline}\label{eqn:tensormap1}
\tri_U^{(\Nbold_1, \dots,\Nbold_r)}: \Ical_{U}|_{n_1}\otimes\Ical_{U}|_{n_2}\otimes\cdots\otimes \Ical_{U}|_{n_r}
\xrightarrow{\triangle_U^{\Nbold_1}\otimes\cdots\otimes \triangle_U^{\Nbold_r}}\\
\to \Hcal_{N_1}(U)\otimes\Hcal_{N_2}(U)\otimes\cdots \otimes\Hcal_{N_r}(U)\to \Hcal_N(U).
\end{multline}

\begin{theorem}\label{thm:generation}
Assume $n>0$. Then the images of the  maps \eqref{eqn:tensormap1} generate $\Hcal_N(U)$. We  already get a set of generators by evaluating these maps on  tensors of the form $(\g_{c_1}-1)|_{n_1}\otimes\cdots\otimes (\g_{c_r}-1)|_{n_r}$, where each $c_i$ is a $1$-cell of $U^*$ and $\g_{c_i}\in \pi_U$ is defined by an orientation of $c_i$.

In particular,  the group $\Htp(U^*,*)$ of self-homotopy equivalences acts on $\Hcal_N(U)$  through its action on $\Ical_{U}|_n$.
\end{theorem} 

This gives rise to a descending filtration $F^\pt\Hcal_N(U)$ of $\Hcal_N(U)$ by taking for  $F^s\Hcal_N(U)$ the  subgroup of
$\Hcal_N(U)$ spanned by the images of the maps   $\Hcal_{N_1}(U)\otimes\Hcal_{N_2}(U)\otimes\cdots \otimes\Hcal_{N_r}(U)\to \Hcal_N(U)$ for which $N_i\not=\emptyset$ and $r\ge s$, so that  
\[
\Hcal_N(U)=F^1\Hcal_N(U)\supset \cdots  \supset F^{n}\Hcal_N(U)\supset F^{n+1}\Hcal_N(U)=0.
\]
So  $F^{n} \Hcal_N(U)$ appears as a quotient of $H^{cl}_1(U)^{\otimes N}$. It follows from Theorem \ref{thm:generation} that $F^k\Hcal_N(U)/F^{k+l}(\Hcal_NU)$  is a quotient of a direct sum of a finite number of copies of  $\Ical_{U}|_{l+1}$.

With the help of the $\tri$-maps, we can also  describe $\underline{\pi\Hcal}(G)$ in terms of  a finite presentation of $G$ (Corollary \ref{cor:kernelmultiplecells-groups}). In the case of only one relation, or equivalently one $2$-cell, we will find a categorical bar complex which allows us to address all homological degrees.

\begin{theorem}\label{thm:generation1}
If $U^*$  is a $2$-complex with a single $2$-cell, then $\Htp(U^*,*)$ acts on $H^{cl}_{n+k}(\conf_N(U))$
through its action on $\Lambda_U|_{n-k}$
\end{theorem}  
 
In this case $\pi_U$ is a one relator group. If furthermore the relation is \emph{primitive} (i.e. not a proper power of another element of the free group), then a theorem of Labute \cite{labute} tells us that $(\Ical_U^{n+1}+1)\cap \pi_U=\pi_U^{(n+1)}$, so that  the kernel of the $\Htp(U, *)$-action  on the quotient $\pi_U|_n=\pi_U{/\pi_U^{(n+1)}}$ is the same as that of its action  on $\Lambda_U|_n$ (or $\Ical_U|_n$ for that matter).
\smallskip

This covers the case when $U$ is a connected oriented open surface  of  finite type. 
So then the kernel of the $\Htp(U, *)$-action on $\pi_U|_n$ is the  same as that of its its action  on $\Lambda_U|_n$ (or $\Ical_U|_n$ for that matter).  The mapping class group $\Mod(U)$ acts on $H^k(\conf_N(U))\cong H^{cl}_{2n-k}(\conf_N(U))$ through $\Htp(U^*,*)$ and  Theorem \ref{thm:generation1} implies that it then  does so via its action on $\pi_U|_{k}$. So the kernel of the latter action, which is known as the \emph{$k$th Johnson group $J^k\Mod(U)$}, acts trivially on  $H^k(\conf_N(U))$.
We derive from this  a similar  result for a closed surface (with a better range for the cohomology in top degree $n+1$, see Corollary \ref{cor:closedsurface}).

When $U$ is a once punctured surface (so that $U^*$ is a closed surface), this is due to Bianchi--Miller--Wilson \cite{bmw}. They also conjectured that for $k$ as above, the group $J^k\Mod(U)$ is exactly the kernel of this action. We  find however that
this is not so: we show that the common kernel $\Ical^\cfg_n(U)\subset\Ical_U|_n$ of the  maps 
$\triangle_U^\Nbold: \Ical_{U}|_n\to  \Hcal_N(U)$ (which is of course $\Mod(U)$-invariant) is nonzero when the genus of $U$ is $\ge 2$ and $n\ge 4$
and we  construct elements of $\Mod(U)$ that act nontrivially on $\Ical_U|_n$ but trivially on $\Ical_U|_n/\Ical^\cfg_n(U)$.
In order to determine the exact kernel, we must determine  $\Ical^\cfg_n(U)$. This will be the subject of a subsequent paper by the second author.

\subsection*{Related work} The first instance  we are aware of concerns the case when $U^*$ is  a compact oriented compact  
 surface of genus $g>0$ with connected  boundary  and $*$ is a boundary point. In this case, we can arrange that the $1$-skeleton $U_1^*$ is a wedge if $2g$ circles and the inclusion $(U_1^*,*)\subset (U^*,*)$  is then a homotopy equivalence. Moriyama \cite{moriyama}  proved that then the  reduced homology of  $\conf_n(U)$  is concentrated in the middle dimension $n$ and that the action of 
 $\Mod(U)$ on  $\Hcal_n(U)$ has the same kernel as its action on $\pi_U|_n$ (note that the first assertion is a special case of Theorem \ref{thm:connectedness};  we will reprove the second assertion below). Subsequently Bianchi--Miller--Wilson \cite{bmw} dealt with the case when 
 $U$ is a once-punctured connected  oriented closed surface of genus $g>0$,  as mentioned above.
 Bianchi--Stavrou \cite{bs} exhibited elements of  $J^{n-1}\Mod(U)$ ($n\ge 2$) that  act nontrivially on $H_n(\conf_n(U))$. 
 
Richard Hain and Cl\'ement Dupont recently obtained results similar to ours, such as a version of Theorem \ref{thm:generation1}. Their  main tool is  mixed Hodge theory. This  leads   them to use cohomology  with rational coefficients and to use the `weight analogue' of the  Johnson filtration. This weight  filtration is in general coarser than the Johnson filtration, but  the two coincide  for a closed  surface $X$.  They informed us  that in that case  they established that if  the genus of $X$ is $\ge 2$ and $n\le 5$,  then the  kernel of $\Mod(X)$ on $H^n(\conf_n(X); \QQ)$ is precisely the Johnson subgroup $J^n\Mod(X)$.
\smallskip

\subsubsection*{Acknowledgements} 
Although our set-up  is somewhat different  from the  Bianchi--Miller--Wilson paper mentioned above,  
we were certainly inspired by it  and via that paper also by Moriyama's.  We thank  Andrea Bianchi and  Dick Hain  for several helpful comments. We are indebted to Louis Hainaut for pointing out an error in an earlier draft.
\smallskip

\subsubsection*{Notational conventions}
A composition  of  $k\ge 2$  paths 
$\g_1,\cdots ,\g_k$ with domain $[0,1]$ without parentheses: $\g_1\g_2\cdots \g_k$,  denotes their \emph{concatenation}, which is by definition their juxtaposition defined on $[0,k]$, followed by rescaling by $1/k$.

If $V$ is a finitely generated abelian  group,  then for every finite set $I$, $V^{\otimes I}$ is defined as usual, namely as a quotient of the free abelian group generated by $V^I$. This quotient is characterized by the universal  property that for every  multilinear map from $V^I$ to an abelian group $V'$ uniquely  factors through a linear map $V^{\otimes I}\to V'$. So $V^\emptyset =\ZZ$ and $V^{I}\otimes V^J$ is canonically identified with $V^{I\sqcup J}$. For an integer $n\ge 0$, $V^{\otimes n}$ stands for $V^{[n]}$. If $I$ is totally ordered and of size $n$, then we have a unique order preserving bijection of $I$ with $\{1,2, \dots ,n\}$ giving  an isomorphism $V^{\otimes I} \cong V^{\otimes n}$. 

We use the boldface font to indicate a totally ordered  set. For example, if a set  $N$ has been endowed with a total order, we may denote it by $\Nbold$.
\medskip

\tableofcontents

\section{Basic properties of our monoidal functor}
We begin with the proof of the first two properties stated in the introduction. Unless mentioned otherwise, $(U^*,*)$ is a pointed  finite complex  having $\{*\}$ as its only $0$-cell.  We also assume that \emph{each $1$-cell $c$ has been oriented}, so that it defines an element  $\g_c\in\pi_U$. {Let $K$ be a finite set.} We construct a  cell decomposition for $U^*{}^K$ which makes  $D_K(U)$ a subcomplex.

\subsection{The lexicographic stratification of $\RR^d$}
We first do this for $U=\RR^d$, so that $\conf_K(\RR^d)$ becomes a union of cells of dimension $\ge |K|+d-1$. We note that Example \ref{example:basic} is the special case of this in $d=1$; for $d=2$, this is known as the Fuchs-Neuwirth stratification.

Firstly, given any $f\in (\RR^d)^K$, i.e.,  a map $f: K\to \RR^d$, we let $K\twoheadrightarrow\overline K$ be the quotient of $K$ such that $f$ factors through an injection of ${\overline{K}}$ in $\RR^d$. Conversely, this quotient must be thought of as  defining the  `diagonal stratum' of $(\RR^d)^K$ that consists of all 
$ f\in (\RR^d)^K$ that factor through an embedding of $\overline K$ in $\RR^d$. Observe that $\conf_K(\RR^d)$ here appears as the stratum  defined by the case $\overline K=K$.

Now, consider the lexicographic total order $\prec$ on $\RR^d$ defined by 
$x\prec y$ if, for some $k\in [d]$, $x_i=y_i$  for all $i=1, \dots, k-1$ and $x_k<y_k$. Then the induced injection    $\bar f: \overline K\hookrightarrow \RR^d$ endows also $\overline K$ with a total order
inherited from $\prec$. This total order has an additional structure that comes from the lexicographical nature of $\prec$, and it is this  structure that defines the cell containing $f$. We obtain it by  taking the  successive images of $K$ under the maps $(f_1, \dots, f_i)$, where $f=(f_1, \dots, f_d)$, i.e. the maps that remember only the first $i$ coordinates of $f$. This gives a sequence of surjections 
\[
\overline K=K_d\twoheadrightarrow K_{d-1}\twoheadrightarrow \cdots K_1\twoheadrightarrow K_0
\]
(where we added the map to the  singleton $K_0=\RR^0$ for convenience) with a total order on each term such that the maps between them are non-decreasing. We call this a \textit{lexicographic structure} on $K$.

We claim that the  collection of $f\in (\RR^d)^K$ giving rise to the same lexicographic structure is in a natural manner a product of open simplices (and hence makes up a cell). In fact, every $a\in K_i$ for all $0\le i\le d-1$ contributes a factor to this product as follows.
The points in the preimage $K_d(a)\subset K_d$ of  $a$ under $K_d\twoheadrightarrow K_i$ must agree in their first $i$-coordinates when mapped by $f$ (the parametrisation of these $i$ coordinates is the job of other simplices). The quotient $K(a)\subset K_{i+1}$ of $K_d(a)$ (which is also the preimage of $a$ under $K_{i+1}\twoheadrightarrow K_i$) groups these points by common $(i+1)$st coordinates: these are given by an order preserving embedding $K(a)\hookrightarrow \RR$. So this defines an open simplex in 
$\RR^{K(a)}$ (of dimension $|K(a)|$) that we denote by $\tri^{K(a)}$.

So the cell thus defined is identified with $\prod_{a\in K'} \tri^{K(a)}$ and has dimension $\sum_{j=1}^d |K_j|$. This cell lies in $\conf_K(\RR^d)$ if and only if $K=\overline K$. It follows that its dimension then lies  between  $|K|+(d-1)$ and $d|K|$ with the minimal value
$|K|+(d-1)$  taken by  the connected components of $\RR^{d-1}\times\conf_K(\RR)\subset  \conf_K (\RR^d)$.

\subsection{Proofs  of  some of the theorems}\label{subsec:celldecomposition} {As before, let $N$ stands for a finite set of size $n$.}
For the proof  of Theorem \ref{thm:connectedness}, all we need to do is to endow the pair $(U^*{}^N, D_N(U))$, for a general $U$, with a relative cell decomposition such that $\conf_N(U)$ becomes a union of cells of dimension $\ge n$. We do this by using the lexicographic stratification.

\begin{proof}[Proof of  Theorems \ref{thm:connectedness} and \ref{thm:generation}]
We start with the product cell structure on $U^*{}^N$ and refine where necessary. Since a  product of cells is  a product of powers of pairwise distinct cells, we  only need to decompose into the cells  a  product $c^K$, where 
$c$ is a  cell in $U$ and $K\subset N$. We must do this in such a  manner  that $\conf_K(c)$ becomes a union of cells, each of  which is of dimension $\ge d-1+|K|$ with equality when $d=1$. 
But then a homeomorphism $\RR^d\xrightarrow{\cong} c$, where $d=\dim c$, which extends to a continuous map $[-\infty, \infty]^d\to U^*$, prescribes a way to decompose $c^K$ according to the lexicographic stratification. All its cells have dimensions  $\ge (d-1)+|K|$ as desired.

This shows at the same time that any $n$-cell in $\conf_N(U)$ is given by a collection of pairwise distinct $1$-cells $c_1, \dots, c_r$, a decomposition 
$N=N_1\sqcup \cdots \sqcup N_r$ into nonempty subsets, and a total order on each of the parts (denoted $\Nbold_i$). This cell is then identified with  
the product of open  simplices of dimensions $|N_1|, \dots, |N_r|$. The element of $\Hcal_N(U)$ it defines is precisely the image of 
$(\g_{c_1}-1)\otimes\cdots \otimes (\g_{c_r}-1)$ under 
$\tri_U^{(\Nbold_1, \dots,\Nbold_r)}$. This proves  that the maps \eqref{eqn:tensormap1} generate $ H_n(U^*{}^N, D_N(U))=\Hcal_N(U)$ and thus establishes Theorem \ref{thm:generation}.
\end{proof}

\begin{proof}[Proof of Corollary \ref{cor:fundgrpdependence}]
The preceding shows that $\Hcal_N(U)$ depends only on the homotopy type of the $2$-skeleton of $U^*$. In order to show that it depends only on the fundamental group $\pi_U$ it suffices to show that taking the wedge of $U^*$ with a $2$-sphere does not affect $\Hcal_N$. If we think of this 
 $2$-sphere  as the Riemann sphere, then this amounts to saying that the inclusion $U\subset U\sqcup \CC$ induces an isomorphism on $\Hcal_N$. 
In view of the monoidal  property  \eqref{eqn:spatial-lmonoidal}, it then suffices to show that $\Hcal_I(\CC)=H^{cl}_{|I|}(\conf_I(\CC))$ is zero 
whenever $I$ is nonempty. By Poincar\'e--Lefschetz duality, we have $H^{cl}_{|I|}(\conf_I(\CC))\cong H^{|I|}(\conf_I(\CC))$. 
As is well-known, the homotopy type of $\conf_I(\CC)$ is a finite cell complex of dimension $|I|-1$ and hence $H^{|I|}(\conf_I(\CC))=0$.
\end{proof}

\subsection{A ring homomorphism}
We let $\tri^n$ be short for $\tri^{1<2<\cdots <n}$. The   \emph{decomposition formula} says that 
if  $\g_0,\g_1$ are composable homotopy classes of arcs in some space, then 
\begin{equation}\label{eqn:decformula}
\textstyle  \tri^n(\g_0\g_1)=\sum_{k=0}^n \tri^k(\g_0)\times \tri^{n-k}(\g_1).
\end{equation}
It merely expresses the fact that if we have    $0=t_0\le t_1\le\cdots\le t_{n+1}=1$, then there exists a $k\in [n]$ such that $t_k\le \half \le t_{k+1}$. 

This suggests that in our situation  we consider 
\[
\textstyle \widehat\Hcal_\pt(U):=\prod_{n=0}^\infty \Hcal_n(U)t^n.
\]
The K\"unneth product makes  this a graded algebra. This algebra was considered earlier by Moriyama in case $U$ is a disjoint union of copies of $\RR$ and so we will refer to $\widehat\Hcal_\pt(U)$ as the \emph{Moriyama algebra} of $U$.

If we  denote the $\Ical_U$-adic completion of the group ring $\Lambda_U=\ZZ\pi_U$ by
\[
\textstyle \hat\Lambda_{U}:=\varprojlim_n \Lambda_{U}|_n,
\]
then the decomposition formula implies that we have a ring homomorphism 
\begin{equation}\label{eqn:}
\textstyle \tri_{U,t}:  \hat\Lambda_U \to \widehat\Hcal_\pt(U), \quad   \tri_{U,t}(\g):=\sum_{n\ge 0} \tri_U^n(\g)t^n.
\end{equation}

\subsection{The case of the real line}\label{subsect:R}
Recall that  $\uHcal(\RR)$ assigns to $N$ the free abelian group generated by the set of total orders on  $N$. The following lemma makes the  description of
$\uHcal(\RR)$ as a monoidal functor complete.

\begin{lemma}\label{lemma:EZformula}
The exterior product is  given by the Eilenberg-Zilber formula
\[
\textstyle \tri^\Ibold\times \tri^\Jbold=\sum_\Kbold \sign(\Ibold\Jbold,\Kbold) \hspace{2pt} \tri^\Kbold, 
\]
where  the sum is over all  shuffles  of $\Ibold$ and $\Jbold$,  that is,  total orders $\Kbold$ on $I\sqcup J$ which extend the 
total orders $\Ibold$ and $\Jbold$, and $\sign(\Ibold\Jbold, \Kbold)$ is the sign of the permutation 
that takes  the concatenation $\Ibold\Jbold$ to $\Kbold$.
\end{lemma}
\begin{proof}
This is indeed a consequence of the  Eilenberg-Zilber formula.
The product  of the  two open simplices given by $s_1<\cdots <s_k$ and $t_1<\cdots <t_l$ is a signed sum of open $(k+l)$-simplices. Each such simplex comes from putting the $s_i$ and $t_j$ in increasing order: this defines a shuffle of $[k]$ and $[l]$ which  takes the sign of the resulting permutation of the lexicographically ordered  $(s_1, \dots ,s_k, t_1,  \dots ,t_l)$.
\end{proof}

We can also state this in terms of the truncated group ring of $\pi_\RR=\pi_1( \PP^1(\RR), \infty)$.  
If $\g_o\in \pi_\RR$ is the positive  generator, then this tells us that 
\[
\textstyle \tri^n_\RR(\g_o\g_o)=\sum_{k=0}^n \tri^k_\RR(\g_o)\times \tri^{n-k}_\RR(\g_o)=\sum_{k=0}^n \tri^k\times \tri^{n-k}
\]
 and more generally that
$\tri^n_\RR(\g_o^r)$ is a noncommutative polynomial in the $\{\tri^k\}_{k=1}^n$ that is homogeneous of degree $n$
if we assign to  $\tri^k$ the weight $k$.

On the other hand, if we put $c_o:=\g_o-1$, then  $\Lambda_\RR|_n =\ZZ[c_o]/(c_o)^{n+1}$ and $\Ical_\RR|_n=(c_o)/(c_o)^{n+1}$.
The above formula shows that the  $\{\g_o^k-1\}_{k=1}^n$ maps to linearly independent elements
of $\Hcal_n(\RR)$. We thus find: 

\begin{corollary}\label{cor:}
The map $\tri^n_\RR: \Ical_\RR|_n\to \Hcal_n(\RR)$ is injective.
\end{corollary}

If  $U$ is a disjoint union of a finite number copies of $\RR$, i.e., of the form $U_E:=E\times \RR$, where $E$ is a finite set, then $U_E^*$ is  a wedge of copies of $\PP^1(\RR)$ with base point $\infty$. If we put a total order on $E$ 
so that it gets identified with $[r]$ for some $r$, then the cells in $\conf_N(U_E)$ are naturally oriented and we get an identification of  
$\uHcal(U_E)$ with the $r$-fold tensor product of
$\uHcal(\RR)$.

\subsection{Degeneracy and truncation}\label{subsection:truncation}
Given distinct $i,j\in N$, denote by $N\twoheadrightarrow N_{ij}$ the quotient of $N$ obtained by identifying these two elements.  
So $N_{ij}$ has size $n-1$. We define a \emph{degeneracy  map} 
\[
\p_{ij}=\p^U_{ij}: \Hcal_N(U)\to \Hcal_{N_{ij}}(U)
\]
(which  will be  symmetric in $i$ and $j$) as follows. 
The surjection $N\to N_{ij}$ induces an embedding  $s_{ij}: U^*{}^{N_{ij}}\to U^*{}^N$ which  parametrizes a hyperdiagonal and thus lands in  
$D_N(U)$.  Let $D^{ij}_N(U)$ denote   the closure of the complement of the image of $s_{ij}$ in $D_N(U)$. 
Then  $D_N(U)$ is of course the union of the image of $s_{ij}$ and $D^{ij}_{N}(U)$. We find that $s_{ij}^{-1}D^{ij}_{N}(U)=D_{N_{ij}}(U)$ so that   $s_{ij}$ defines  a relative homeomorphism  of the pair $(U^*{}^{N_{ij}},D_{N_{ij}}(U))$ onto the pair $(D_N(U), D^{ij}_N(U))$. We then let
$\p_{ij}: \Hcal_{N}(U)\to \Hcal_{N_{ij}}(U)$ be the connecting homomorphism  of the  long exact homology sequence for the triple 
 $(U^*{}^N,D_N(U), D^{ij}_N(U)) $:
\begin{multline}\label{eqn:deg}
 \cdots \to H_{n}(U^*{}^N, D_N(U)))\xrightarrow{\p_{ij}}  H_{n-1}(D_N(U),  D^{ij}_N(U)))\to\\
 \to  H_{n-1}(U^*{}^N,  D^{ij}_N(U))\to\cdots
\end{multline}
This  map is functorial in $U$ and so we have for every $\g\in \pi_U$ a commutative diagram
\[
\begin{tikzcd}
\Hcal_N(\RR)\arrow[r, "\Hcal_N(\g)"]\arrow[d, "\p^\RR_{ij}"] & \Hcal_N(U)\arrow[d, "\p^U_{ij}"]\\
\Hcal_{N_{ij}}(\RR)\arrow[r, "\Hcal_{N_{ij}}(\g)"] & \Hcal_{N_{ij}}(U).
\end{tikzcd}
\]
If $N$ is equipped with a total order $\Nbold$, then (by definition) the top arrow sends the basis element $\tri^\Nbold\in \Hcal_N(\RR)$ to $\tri^\Nbold_U(\g)$. 
Similarly, the evaluation of $\partial_{ij}$ on $\tri_U^{\Nbold_1}(\g_1)\times \cdots\times \tri_U^{\Nbold_r}(\g_r)\in \Hcal_N(U)$ is determined by the evaluation of $\partial_{ij}$ on $\tri^{\Nbold_1}\times \cdots\times \tri^{\Nbold_r}$, where $\g_1, \ldots, \g_r\in \pi_U$, and $\Nbold_1,\ldots, \Nbold_r$ is a partition of $N$ with a total order on each part. It is clear that if the pair  $(i,j)$  appears consecutively in the same part of the partition, (i.e.,  if one is a successor of the other for this total order) we have an induced partition $\Nbold_{1,ij}, \ldots, \Nbold_{r, ij}$ on the set $N_{ij}$ with a total order on each part.

\begin{lemma}\label{lemma:onetrunccomptri}
If the pair $\{i,j\}$ appears consecutively in the same part of the partition $\Nbold_1,\ldots, \Nbold_r$, and $pr(\{i,j\})$ is the number of predecessors of that pair in the concatenated total order $\Nbold_1\Nbold_2\cdots\Nbold_r$, then put 
$\eps_{ij}(\Nbold_1,\Nbold_2,\ldots,\Nbold_r):=(-1)^{{pr(\{i,j\})}}$; otherwise  put  $\eps_{ij}(\Nbold_1,\Nbold_2,\ldots,\Nbold_r):=0$. Then 
$$\p_{ij}(\tri^{\Nbold_1}\times \cdots \times \tri^{\Nbold_r})= \eps_{ij}(\Nbold_1,\Nbold_2,\ldots,\Nbold_r)\tri^{\Nbold_{1,ij}}\times \cdots \times \tri^{\Nbold_{r,ij}}$$ 
(so read zero if the pair $\{i,j\}$ appears nonconsecutively or in different parts of the partition). 
\end{lemma}

\begin{proof}
In the case $r=1$, assume that  $\Nbold_1$ is, without loss of generality, the order $1<2<\cdots <n$. Then the interior of $\tri^{\Nbold}$ is the subset of $\RR^n$
given by $t_1<t_2<\cdots <t_n$ and the locus $t_i=t_j$ appears as a  codimension one face only if $i,j$ is a consecutive pair; in that case it is the $(pr(\{i,j\})+1)$-th face and acquires the claimed sign. For $r>1$, we only get a contribution from the boundary of a simplex containing $i$ and $j$ consecutively; in that case, the Leibniz rule gives us the correct sign.
\end{proof}

\begin{corollary}\label{cor:}
The degeneracy  map $\p_{ij}: \Hcal_N(U)\to \Hcal_{N_{ij}}(U)$ is onto.
\end{corollary}
\begin{proof}
We first observe that every numbered partition of $N_{ij}$ endowed with a total order on each part, determines such a structure for $N$: a part of $N$ is a  preimage of a part on $N_{ij}$ and such a part inherits a  total order from its image in $N_{ij}$, where for the part containing the pair $\{i,j\}$,  we stipulate that  $j$ is the successor of $i$. Then 
each generator of  $\Hcal_{N_{ij}}(U)$ provided by Theorem \ref{thm:generation}, is the image of a generator it provides for $\Hcal_{N}(U)$.
\end{proof}

By taking $N=[n]$ and  $\{i,j\}=\{n-1,n\}$, we get  a map $\p_{n-1,n}: \Hcal_n(U)\to \Hcal_{n-1}(U)$. 
We call this a \emph{truncation map}, because of the following lemma.

\begin{lemma}\label{lemma:truncation}
The diagram of truncations
\begin{equation}\label{eq:truncations}
    \begin{tikzcd}
    \Ical_{U}|_{n}\dar[two heads]\rar & \Hcal_{n}(U)\arrow[d, two heads, "\p_{n-1,n}"]\\
    \Ical_{U}|_{n-1}\rar & \Hcal_{n-1}(U)
\end{tikzcd}
\end{equation}
commutes up to the sign $(-1)^{{n}}$.
\end{lemma}
\begin{proof}
The horizontal maps are induced by $\g\in \pi_U\mapsto \tri^n(\g)$ resp.\ $\tri^{n-1}(\g)$. 
Since the pair $(n-1,n)$ has $n-2$ predecessors, Lemma \ref{lemma:onetrunccomptri}  implies that  $\p_{n-1,n}(\tri^n)=(-1)^n \tri^{n-1}$ and so the left vertical map is given by restriction.
In terms of the group rings, this  restriction amounts to a truncation.
\end{proof}

If $k\le n$, we may iteratively apply the degeneracy maps $\partial_{n-1,n}, \partial_{n-2,n-1}, \ldots,$ $\partial_{k,k+1}$ to get a map surjective map $\partial^k_n:\Hcal_n(U)\to \Hcal_k(U)$. 
\begin{lemma}\label{lem:iterateddegeneracy}
    The iterated degeneracy $\Hcal_n(U)\to \Hcal_k(U)$
    annihilates the subspace $F^{k+1}\Hcal_n(U)$.
\end{lemma}
\begin{proof}
    From Lemma \ref{lemma:onetrunccomptri}, the only way this iterated degeneracy might not vanish on an element 
    $\tri^{\Nbold_1}(\g_1)\times \cdots\times \tri^{\Nbold_r}(\g_r)$ is if all elements $k,k+1, \ldots, n$ appear in the same part of the partition $N_1,\ldots, N_r$. On the other hand, the space $F^{k+1}\Hcal_n(U)$ is spanned by elements $\tri^{\Nbold_1}(\g_1)\times \cdots\times \tri^{\Nbold_r}(\g_r)$ with $r\ge k+1$ and with each $N_1,\ldots, N_r$ non-empty. This forces the size of each part $N_i$ to be $\le n-r+1\le n-k$ so that, in particular, it cannot contain the $n-k+1$ elements $k,k+1, \ldots, n$. The conclusion follows.
\end{proof}

\begin{remark}\label{rem:}
If $U$ is a Riemann surface of finite type, then the degeneracy map can be understood as a residue map (we are taking the residue of a logarithmic  $n$-form on $\conf_N(U)$ along the hyperdiagonal $z_i-z_j$). This is a good mnemonic for the sign rules 
when we consider iterated degeneracies. Suppose $\{i,j\}$ and $\{k, l\}$ are \emph{distinct} unordered pairs in $N$. Then the associated  quotients  $N\to N_{ij}$ and  $N\to N_{kl}$ fit in pushout diagram
\[
 \begin{tikzcd}
    N\dar\rar &  N_{ij}\dar\\
   N_{kl}\rar &  N_{ij, kl}
\end{tikzcd}
\]
If the two pairs are disjoint, then  $\p_{ij}\p_{kl}=-\p_{kl}\p_{ij}$. Otherwise, we can assume that the second pair is $\{j, k\}$ with $i,j,k$ pairwise distinct.
In that case $\p_{i[jk]}\p_{jk}+\p_{j[ki]}\p_{ki}+\p_{k[ij]}\p_{ij}=0$, where $[ij]$ stands for the image of $\{i,j\}$ in $N_{ij}$. These operations appear implicitly  in the work of Cohen-Taylor \cite{cohen_taylor}.
\end{remark}

\subsection{The cellular  chain complex}

In the proof of Theorem \ref{thm:connectedness} we put a total ordering  on each  cell of $U$. This orients that cell. We want to orient the cells of $\conf_N(U)$ as well.  If $c$ is a cell in $U$, then
given  $f\in\conf_K(c)$, the map $f:K\to c$ is injective. The iterated   lexicographical  total order on $c$ (with as many levels as the dimension of $c$) determines then also such a structure on $K$. The cell which contains $f$ is the connected component of the set of all $f'\in\conf_K(c)$ that give rise to the same such structure on $K$. Note that this structure also  determines  an orientation of that cell. A cell of $\conf_N(U)$
is a product of cells of this type  where the product is over a  collection of pairwise distinct cells of $U$. Hence an orientation of this 
cell is obtained by imposing a total order on the collection of the cells of $U$. We make this assumption from now on so that each cell determines a generator of the cellular chain complex 
\[
\Cscr_N(U)_\pt:=\Cscr_\pt( U^*{}^N, D_N(U)).
\]
This complex is zero in degree $<n$ and its  homology is that of  the homology with closed support of  $\conf_N(U)$.
The assignment $N\mapsto \Cscr_N(U)_\pt$ defines a monoidal functor on $\FB$ that takes values  in the category of bounded below chain complexes.

\subsection{Attaching a $2$-cell}\label{subsec:attaching2cell}
In view of the  central role  of $\uHcal(\RR)$ in what is going to follow, we often suppress its  argument and write
\begin{equation}\label{eqn:suppressR}
\Hcal_N:=\Hcal_N(\RR), \quad \uHcal:=\uHcal_N(\RR),  
\end{equation}
where we trust  the reader not to confuse this with the different use of this symbol in \eqref{eqn:2functor}.
In this section we assume that we are given a $2$-cell $e$ of $U$ such that $U_o:=U\ssm \{e\}$ is closed in $U$. In other words, $U^*$ is obtained 
from $U^*_o$ by attaching a single $2$-cell. We are going to  express the closed support homology of the configuration spaces of $U$ in terms of this of $U_o$.

It will be convenient to assume that the $2$-cell $e$ is parametrized by the upper half plane 
$\HH_+\subset \CC\subset \PP^1(\CC)$ (rather than the interior of the standard $2$-disk). This has  $\PP^1(\RR)$ as its boundary.  
We assume that the  attaching  map $att: \PP^1(\RR)\to U_o^*$ takes $\infty$ to $*$ so that $U^*$ can be identified with  
$\overline\HH_+\cup_{att} U_o^*$. As this discussion is only about the homotopy type of $(U^*,*)$ and the latter  depends only on the image of 
$\att$ in $\pi_{U_o}$ (which we shall denote by $\zeta$), we may (and will) assume that $att$ maps to the $1$-skeleton of $U_o^*$. 
It is clear that  the inclusion $U^*_o\subset U^*$ induces a surjection $\pi_{U_o}\to \pi_U$ whose kernel is the normal subgroup generated by $\zeta$
(and hence  the induced ring homomorphism  $\Lambda_{U_o}\to \Lambda_{U}$ is surjective with kernel generated by the 2-sided ideal generated by $\zeta-1$).
It is also clear that $\Cscr_N(U_o)_\pt$ appears here as a subcomplex of $\Cscr_N(U)_\pt$.  

To make this precise,  we endow $e$ with a lexicographical ordering via its parametrization by the upper half plane, by first considering 
the imaginary part: $x+\sqrt{-1} y\prec x'+\sqrt{-1}y'$ if $y<y'$ or $y=y'$ and $x<x'$ (but use the standard counterclockwise orientation of $\HH_+$ to orient $e$).  So given a finite set $K$, then a cell of $\conf_K(e)$ is 
given by a total order on $K$ that  has a lexicographical structure with two levels:  a decomposition $K=K_1\sqcup\cdots \sqcup K_r$ into 
nonempty subsets and a total order $\Kbold_i$ 
on each $K_i$. This cell consists of $f: K\to \HH_+$ with $\im (f)$ constant on $K_i$ with a value smaller than its value on $K_{i+1}$ and 
with $\re(f)|K_i$ order preserving.  We thus find an identification  of modules (not of complexes!)
\begin{multline*}
\Cscr_N(U)_\pt=\oplus_{N_0\subset N}\, \Cscr_{N\ssm N_0}(\HH_+)_\pt\otimes \Cscr_{N_0}(U_o)_\pt=\\
=\oplus_{r\ge 0}\oplus_{(N_0, N_1, \dots ,N_r)}\; \Hcal_{N_r}\otimes\cdots\otimes  \Hcal_{N_1}\otimes\Cscr_{N_0}(U_o)_\pt,
\end{multline*}
where  in the last line the sum is over all decompositions $N=N_0\sqcup N_1\sqcup \cdots N_r$ into subsets ($r=0,1,\dots$) with $N_i$  \emph{nonempty} when $i>0$.  To make this an identification of \emph{graded} modules, we should replace each tensor factor $\Hcal_{N_i}$ by its K\"unneth product with the fundamental class of the positive imaginary axis, i.e., the natural generator of  element of $H^{cl}_1(\sqrt{-1}\RR_{>0})$. This has the effect of putting  $\Hcal_{N_i}$  in degree $|N_i|+1$; we indicate this by writing 
$\Hcal_{N_i}[1]$ instead (\footnote{More generally, if $\Bcal_\pt$ is a \emph{homological} complex, then $\Bcal[m]_\pt$ is 
defined by $\Bcal[m]_k=\Bcal_{k-m}$ and the boundary operator is multiplied with  $(-1)^m$.
This is equivalent with  taking the tensor product of $\Bcal_\pt$ with the single term complex $\ZZ$ placed in degree $-m$. The passage $\Ccal^k:=\Bcal_{-k}$ gives the corresponding cohomological  convention: $\Ccal[m]_k=\Ccal_{m+k}$.}). The summand  associated with  $r=0$ (for which  $N_0=N$) then gives  the subcomplex $\Cscr_N(U)_\pt$. For reasons which become clear, it makes good sense to switch to the bar notation insofar tensor products of the $\uHcal$ factors are involved:
\[
\Cscr_N(U)_\pt=\oplus_{r\ge 0}\oplus_{(N_0, N_1, \dots ,N_r)}\; \Hcal_{N_r}[1]\, \big|\, \cdots  \, \big|\, \Hcal_{N_1}[1]\otimes\Cscr_{N_0}(U_o)_\pt,
\]
or in  functorial terms:
\[
\underline{\Cscr} (U)_\pt=  \cdots \, \big|\, \uHcal [1]\, \big|\, \uHcal [1] \, \big|\, \uHcal [1]\otimes\underline{\Cscr}(U_o)_\pt.
\]

Since $\Hcal_{K}$ comes with the basis 
$\{\tri^{\Kbold}\}$, indexed by the total orders on $K$, we can   also write this as 
\[
\Cscr_N(U)_\pt=\oplus_{(N_0, \Nbold_1,\Nbold_2 \dots ,)}\; \big[\cdots  \, |\, \tri^{\Nbold_2}[1] \, |\, \tri^{\Nbold_1}[1]\big]\otimes\Ccal_{N_0}(U_o)_\pt,
\]
where the index set of the summation is implied by the preceding (and use the customary bracket notation for `barred' arguments). 

We  shall  see that  the  boundary map turns this complex (up to signs) into a  `reduced bar complex' in which the role of the differential graded algebra is taken by
the monoidal category $\uHcal$ (with zero differential) and $\uHcal(U_o)$ is considered a $\uHcal$-module via $att$.
This boundary map involves  K\"unneth products, which,  because of  the imposed shifts,  are now all of degree $-1$:
\begin{gather}
\times: \Hcal_{K''} [1] \, \big|\, \Hcal_{K'}[1]\to \Hcal_{K} [1], \label{eqn:times}\\
\times_{\zeta}: \Hcal_{K''}[1]\otimes \Ccal_{K'}(U_o)_\pt\to \Ccal_{K}(U_o)_\pt, \quad \xi\otimes \psi\mapsto \tri^{K''}_{U_o}(\xi)\times\psi, \label{eqn:timeszeta}  
\end{gather}
where in either case $K:=K'\sqcup K''$. The signs in the Eilenberg-Zilber formula of Lemma \ref{lemma:EZformula} 
$\tri^\Ibold\times \tri^\Jbold=\sum_\Kbold \sign(\Ibold\Jbold,\Kbold) \hspace{2pt} \tri^\Kbold$ 
come from the sign of permutations:  a shuffle of  $(x_1,\dots , x_k)$ and $(x_{k+1},\dots , x_{k+l})$ defines a linear transformation of $\RR^{k+l}$
and the determinant of this permutation accounts for the sign that appears here. Here we shuffle points of $\HH_+$ with the same imaginary part 
and so this does not affect the  Eilenberg-Zilber formula: it remains valid  in the context of \eqref{eqn:times} and \eqref{eqn:timeszeta}.
This  gives for $t=1,2,\dots$ maps
\begin{gather*}
\times^{(t)}:\Hcal_{N_r}[1]\, |\cdots |\, \Hcal_{N_1}[1]\to \Hcal_{N_r}[1]\, |\cdots |\,\Hcal_{N_{t+1}\sqcup N_t}[1]\, |\cdots |\, \Hcal_{N_1}[1],\\
\times^{(0)}_{\zeta}: \big[\Hcal_{\Nbold_r}[1]\, |\cdots |\, \Hcal_{\Nbold_1}[1]\big]\otimes  \Cscr_{N_0}(U_o)_\pt\to\big[\Hcal_{\Nbold_r}[1]\, |\cdots |\, \Hcal_{\Nbold_2}[1]\big]\otimes  \Cscr_{N_1\sqcup N_0}(U_o)_\pt
\end{gather*}
given  by  the K\"unneth products  \eqref{eqn:times} resp.\ \eqref{eqn:timeszeta}.
What changes however, are the signs when we regard these maps as components of a boundary operator, for we then must invoke the 
the Koszul rule. This  gives the  following proposition.

\begin{proposition}\label{prop:barcomplex}
The boundary map of the complex $\Cscr_N(\HH_+)_\pt$ (whose homology in degree $k$ is 
$H^{cl}_k(\conf_N(\HH_+))$ is equal to 
\[
\textstyle \p_+:= \sum_{t\ge 1} (-1)^{(1+n_r)+\cdots +(1+n_{t+2})}\times^{(t)}
\]
and the  boundary map  $\p$ of the complex $\Cscr_N(U)_\pt$ (whose homology in degree $k$ is $H^{cl}_k(\conf_N(U)$) is equal to 
\[
\textstyle \p:= \p_+\otimes 1+\times^{(0)}_{\zeta}=\sum_{t\ge 1}  (-1)^{(1+n_r)+\cdots +(1+n_{t+2})}\times^{(t)}\otimes 1+(-1)^{n-n_1+r-1}\times^{(0)}_{\zeta}
\]
\hfill $\square$
\end{proposition}

Noteworthy is the case  of degree $n$, as it gives  a presentation of $\Hcal_N(U)$ via the exact sequence
\begin{equation}\label{eqn:degree_n_pres}
\oplus_{\emptyset\not=I\subset N}\Hcal_I[1]\otimes \Cscr_{N\ssm I}(U_o)\xrightarrow{\times_\zeta} \Cscr_{N}(U_o)\to \Hcal_N(U)\to 0.
\end{equation}

Iterated application of this proposition gives us control on the total closed support  homology of the configuration spaces in case $U^*$ is a $2$-complex. Since  the functor $\uHcal_U$ depends only on  the $2$-skeleton of $U$, this gives us a recipe to compute that functor in the general case.

\begin{corollary}\label{cor:zeta-generation}
Assume $\dim U_o=1$,  so that $U_o^*$ is the  $1$-skeleton of $U$. If  a homotopy class $f_o\in \Htp(U_o,*)$ fixes the element  $\zeta\in \pi_{U_o}$ defined by the attaching map,  then $f_o$ extends to a  $f\in \Htp(U,*)$ whose action on  $H^{cl}_k(\conf_N(U))$ is  through its action  on $\Ical_{U}|_{n-k}$.
\end{corollary}
\begin{proof}
In this case, $\Cscr_K(U_o)_\pt$ is a single term complex concentrated in degree $|K|$ and then equal to $\Hcal_K(U_o)$.

Since $f_o$ fixes $\zeta$, it can be represented by a continuous map $F_o:U_o\to U_o$ together with a homotopy between $\att$ and $F_o\circ att$. 
Perform this homotopy (relative to the base point) on the closure of the strip  $0<\im(z)\le 1$ of $\HH_+$ in $U$ and extend it by the map $z\mapsto z-\sqrt{-1}$ on $\im(z)\ge 1$ to obtain the cellular map $F: U\to U$. 
By construction, the map $F$  preserves the total order we imposed on $\HH_+$. These conditions are precisely what we need so that $F$ acts cellularly on the pair  $(U^{*N}, D_N(U))$.  It then acts functorially on  $\Cscr_N(U)_\pt$ in a way that respects the monodial structure. On the subcomplex $\Cscr_{N_0}(U)_\pt=\Hcal_{N_0}(U_o)$ the action is via $f_o$. 
The action of $F$ on the other  tensor factor $\Hcal_{N_i}[1]$ (which defines   in general not a subcomplex) is trivial: this group  coincides with 
$\Cscr_N(\HH_+)_\pt$ on which $F$ acts through the quotient $U/U_o\cong (\HH_+)^*$, but on this quotient 
$F$ is homotopic to the identity via the lexicographic maps $H_t:z\in \HH_+^*\mapsto F(z+\sqrt{-1}t)\in \HH_+^*$. It is clear that the induced action of $F$ on the  homology $H^{cl}_k(\conf_N(U))$ is that of $f$. 

The  action  of $f_o$ on $\Hcal_{N_0}(U_o)$ is through its action on  $\Ical_{U_o}|_{n_0}$ as a ring automorphism which  fixes the image of $\zeta$. Recall that this invokes (via Theorem \ref{thm:generation}) the K\"unneth products of  $f_o$-equivariant maps 
$\tri^\Ibold: \Ical_{U_o}|_i\to \Hcal_I(U_o)$ with $I\subset N_0$ and $i=|I|$.  
Hence  the same is true for its action  $\Cscr_{N}(U)_{n+k}$,  where we must take   $i=|I|\le n-k$. 

On the other hand,  we know that $f_o\in \Htp(U_o,*)$ acts on $H^{cl}_{n+k}(\conf_N(U))$ through its image 
$f\in \Htp(U,*)$. It is  geometrically  clear how: the maps $\tri^\Ibold: \Ical_{U_o}|_i\to \Hcal_I(U_o)$ involved in $\Cscr_{N}(U)_{n+k}$,  will after post composition  with the natural map  $\Hcal_I(U_o)\to \Hcal_I(U)$ factor through $\tri^\Ibold: \Ical_{U}|_i\to \Hcal_I(U)$. (More formally,  the image of  
of $\zeta-1 $ in  $\Ical_{U_o}|_i$ is mapped by $\tri^\Ibold$ to a  boundary in  the complex $\Cscr_{N}(U)_\pt$.)
This proves that $f$ acts on $H^{cl}_{n+k}(\conf_N(U))$ via its action on $\Ical_{U}|_{n-k}$.
\end{proof}

\subsection{Attaching several $2$-cells}
Suppose that $U^*$ is a finite $2$-complex, so obtained from its $1$-skeleton $U_1^*$ by attaching the $2$-cells $e_1,\ldots, e_m$ along the words 
$\zeta_1,\ldots, \zeta_m\in \pi_{U_1}$, respectively. Then the relative cellular decomposition of the pair $(U^{*N}, D_N(U))$ 
has lowest cells in degree $n$. These are precisely those in $\conf_N({U_1})$ and so $\Cscr_N(U)_n\cong \Hcal_N(U_1)$ as before. 
In the next dimension, the discussion on lexicographic stratification tells us that each $(n+1)$-cell of 
$\conf_N(U)$ involves exactly one of the $2$-cells $e_1,\ldots, e_m$; it is, in fact, a product of a smallest cell of $\conf_I(e_i)$ and a cell of $\conf_{N\ssm I}(U_1)$ for some  nonempty subset $I\subset N$. These are exactly the types of cells involving `one bar' in the bar complex of Section \ref{subsec:celldecomposition}, except we now have $m$ choices of the cell containing this bar. Our discussion on differentials above then repeats verbatim to this case too, giving the analogue of exact sequence \eqref{eqn:degree_n_pres}:
\begin{equation}
    \oplus_{j=1}^m\oplus_{\emptyset\not=I\subset N}\Hcal_I[1]\otimes \Cscr_{N\ssm I}(U_1)\xrightarrow{\oplus_{j=1}^m\times_{\zeta_j}} \Cscr_{N}(U_1)\to \Hcal_N(U)\to 0.
\end{equation}
If we view the words $\zeta_j$ as maps from $(\PP^1(\RR), \infty)\to (U_1^*,*)$, this can be categorically summarised as follows.

\begin{theorem}\label{thm:kernelmultiplecells}
The kernel  of the surjective $\FB$-morphism $\uHcal(U_1)\to \uHcal(U)$ is the $\FB$-ideal $\underline{\Kcal}(U,U_1)$ generated by the images of the morphisms $\uHcal(\zeta_j):\uHcal\to \uHcal(U_1)$ for $j=1,\ldots, m$ evaluated on all non-empty sets. Concretely, 
${\Kcal}_n(U,U_1)\subset \Hcal_n(U_1)$ is spanned by the   $\Scal_n$-orbit of $\sum_{j=1}^m\sum_{k=1}^n\tri^k_{U_1}(\zeta_j)\times \Hcal_{n-k}(U_1).$
\end{theorem}

This has an immediate implication on a group $G$ presented by a homomorphism  $\zeta: F_m\to F_n$ of free groups. Note that $\underline{\pi\Hcal_N(F_n)}\cong \uHcal(\RR)^{\otimes n}$. Let $\zeta_j\in F_n$ denote the image of the $j$th generator of $F_m$.

\begin{corollary}\label{cor:kernelmultiplecells-groups}
The  functor $\underline{\pi\Hcal(G)}$ is the quotient of  $\underline{\pi\Hcal}(F_n)$ by the $\FB$-ideal $\underline{\Kcal}(\zeta)$ generated by the image of  $\uHcal(\zeta): \uHcal(F_m)\to \uHcal(F_n)$ evaluated on all non-empty sets. Concretely, 
${\Kcal}_n(\zeta)$ is spanned by the   $\Scal_n$-orbit of $\sum_{j=1}^m\sum_{k=1}^n\tri^k(\zeta_j)\times 
\pi\Hcal_{n-k}(F_n)$.
\end{corollary}

\subsection{Examples involving the  upper half plane} A simple way to  verify that a closed chain on a configuration space of $\HH_+$  is not a coboundary is by showing that its intersection product with some cycle of complementary dimension is nonzero. 
We can do this with the help of the `planetary system' model:  take for this cycle  a torus of the following type. Choose $\varepsilon\in (0,1)$ and $z_1\in \CC$ and let  $\Tcal_{z_1,\varepsilon}$ be the set of $z\in \CC^K$ with 
$|z_{i+1}-z_i|=\varepsilon^{i+1}$ for   $i=1,\dots k-1$. This is an oriented $(k-1)$-torus contained in 
$\conf_k(\CC)$.  If we take $\im (z_1)>\varepsilon$,  then this torus lies in  $\conf_k(\HH_+)$. Its   isotopy class does not depend  on $z_1$ and $\varepsilon$, so that we  get a well-defined homology class $\tau_{k-1}\in H_{k-1}(\conf_k(\CC))$. The $\Sfrak_k$-orbit of this class is  known to generate  $H_{k-1}(\conf_k(\CC))$ and hence a cocycle which vanishes on this orbit must be coboundary.
Here are a few simple, yet  basic  examples. 

\begin{examples}\label{examples:basic}

 (i)  Recall that  $[\tri^{(1,2,\dots, k)}]\in \Hcal_k[1]$  defines  in $\HH_+^k$  a $(k+1)$-chain with closed support; it is defined by   the locus   of $(z_1,\dots z_k)$ for which $\im(z_1)=\cdots= \im(z_k)>0$ and 
$\re (z_1)\le\cdots\le  \re(z_k)$. Its  boundary lies on the union of the the loci where $z_i=z_{i+1}$ for some $i$ and hence  is  a closed support cycle when restricted to $\conf_k(\HH_+)$. If we  think of this in cohomological terms (as `taking the transversal intersection product with this cycle') it then  defines a class in $H^{k-1}(\conf_k(\HH_+))$. Its intersection product with the torus $\Tcal_{z_1,\varepsilon}$ 
(with $\im (z_1)>\varepsilon$)  is transversal with 
 $(z_1, z_1+\varepsilon^2,\dots, z_1+\varepsilon^2+\cdots +\varepsilon^k)$  as its only point of intersection.
We thus  find that $[\tri^{(1,2,\dots, k)}]\cap \tau=1$.
 
 For any  permutation $\sigma\in \Sfrak_k$ we get another such  cohomology class
 $\sigma_*[\tri^{(1,2,\dots, k)}]=\sign(\sigma)[\tri^{(\sigma(1), \dots , \sigma(k))}]$.   It is well-known that these classes generate all of $H^{k-1}(\conf_k(\HH_+))$.

(ii)  The closed 3-chain  $[\tri^{(2)}|\tri^{(1)}]$ on $\conf_2(\HH_+)$ describes  the locus   of 
$(z_2,z_1)$ for which $\im(z_1)\le  \im(z_2)$. Its boundary in $\conf_2(\HH_+)$ is where 
$\im (z_1)=\im (z_2)$  and this has  two components according to whether  $\re(z_1)$ is smaller or greater that $\re(z_2)$. Hence
\[
\p_+[\tri^{(2)}|\tri^{(1)}]=[\tri^{(1,2)}]- [\tri^{(2,1)}]
\]
and so $[\tri^{(2,1)}]$ and $ [\tri^{(1,2)}]$ define the same class. The previous example shows that this class is nonzero.

(iii)  The chain $[\tri^{(3)}|\tri^{(1,2)}]$ describes a closed $5$-chain in $\conf_3(\HH_+)$, given as  the locus   of 
$(z_3,z_1, z_2)$ for which $\im(z_1)=\im (z_2)\le  \im(z_3)$ and  $\re(z_1)\le \re (z_2)$. The boundary of this $5$-chain meets  $\conf_3(\HH_+)$ in the locus  where 
$\im (z_2)=\im (z_3)$. It has three components corresponding to three shuffles: 
\begin{equation}\label{eqn:coboundary}
\p_+[\tri^{(3)}|\tri^{(1,2)}]=[\tri^{(3,1,2)}] -[\tri^{(1,3,2)}]+[\tri^{(1,2,3)}].
\end{equation}
Example (i) shows that  each term of the  right hand side of the identity \eqref{eqn:coboundary} represents a nonzero cohomology class in degree $2$ and hence  these three classes have zero sum. 
\end{examples}

\section{The surface case}
In this section $X$ is a closed connected  oriented surface whose genus we denote by $g$. We are also given a nonempty finite subset $P$ and take 
$U:=X\ssm P$. So $U^*$ can be regarded as the quotient of $X$ obtained by identifying  $P$ with a singleton. We number the elements of $P$ as $p_0,p_1,...,p_l$ (so $P$ has size $l+1\ge 1$). The mapping class group $\Mod(U)$ is the connected component group of the group of orientation preserving homeomorphisms  of $X$ which preserve $P$ (not necessarily pointwise). It acts on the  homotopy type  of $(U^*,*)$, giving a  group homomorphism $\Mod(U)\to \Htp(U^*,*)$. By Theorem \ref{thm:generation},  $\Mod(U)$ therefore acts on $\Hcal_N(U)=H_n^{cl}(\conf_N(U))\cong H^n(\conf_N(U))$ through $\Ical_U|_n$. So if we denote the kernel of the $\Mod(U)$-action on $\Hcal_N(U)$ by $J_\cfg^n\Mod (U)$, then 
by the theorem of Labute \cite{labute} mentioned in the introduction,  this is equivalent to saying that  $J_\cfg^n\Mod (U)$ contains the Johnson subgroup
$J^n\Mod (U)$.

\subsection{A cell structure on the  surface}\label{sec:celldecomp of X}
We begin with fixing an oriented  cell decomposition of $X$ that has the nonempty set $P$ as $0$-skeleton.
We let the $1$-cells be defined by arcs  $\beta_1, \dots, \beta_l$ in $X$ with $\beta_j$ connecting $p_{j-1}$ with $p_j$ and oriented loops $\{\alpha_{\pm i}\}_{i=1}^g$ based at $p_0$ ensuring that their relative interiors are pairwise disjoint. Their union,  which we denote  by $X_1\subset  X$,  is the $1$-skeleton of our CW structure and we assume that  these arcs  have been chosen and numbered in such a manner that  $X\ssm X_1$ is  a single $2$-cell $e$ and that $X$ is obtained by attaching a $2$-disk whose attaching map traverses the concatenation
\begin{equation}\label{eqn:concat}
(\alpha_1,\alpha_{-1})...(\alpha_g,\alpha_{-g})\beta_1\beta_2...\beta_l\beta_l^{-1}...\beta_2^{-1}\beta_1^{-1}.
\end{equation}
We assume that $e$ is parametrized by the upper half plane $\HH_+\subset \CC\subset \PP^1(\CC)$  as before, so that the attaching  map 
$att: \PP^1(\RR) \to X_1$  sends $\infty$ to $p_0$  and is, up to an orientation preserving reparametrization,  the  concatenation \eqref{eqn:concat}.
We denote the oriented $1$-cells defined by $\alpha_{\pm i}$ and  $\beta_j$ by   resp.\  $a_{\pm i}$ and $b_j$ and totally order them as 
$(a_1, a_{-1}, \dots, a_g, a_{-g}, b_1, \dots, b_l)$.

This cell decomposition  also determines one  of $U^*=X/P$ such that  $U_1^*:=X_1/P$ is its $1$-skeleton (a wedge of $2g+l$ circles). We shall  identify the positive dimensional cells of $X$ with those of $U^*$. This gives us a presentation of the fundamental groups 
$\pi_{U_1}$ and $\pi_U$. We will make no notational distinction between the paths $\alpha_1, \alpha_{-1}, \dots, \alpha_g,\alpha_{-g}, \beta_1, \dots, \beta_l$ and their images in $\pi_{U_1}$ and so the latter is freely  generated by these elements. The attaching map defines the commutator
\begin{equation}\label{eqn:defzeta}
\zeta:=(\alpha_1, \alpha_{-1})\cdots (\alpha_g, \alpha_{-g})\in \pi_{U_1}.
\end{equation}
It is clear that  $\pi_{U^*}$ can be identified with the quotient of $\pi_{U_1}$ by the normal subgroup $N\la\zeta\ra\vartriangleleft\pi_{U_1}$ generated by $\zeta$. 
 
\begin{lemma}\label{lemma:stabilizer_p_0}
The stabilizer {$\Mod(U)_{p_0}$} of $p_0$ in  $\Mod(U)$ contains $J^1\Mod(U)$ and each of its elements  lifts to an automorphism of  $\pi_{U_1}$ which fixes $\zeta$.
\end{lemma}
\begin{proof}
The boundary map of the homology sequence of the pair $(X,P)$ defines a surjection $H_1(X,P)\to \tilde H_0(P)$ and hence the action of $\Mod(U)$ on $P$ factors through its action on $H_1(X,P)=H_1^{cl}(U)$. In particular, the 
$\Mod(U)$-stabilizer of $p_0$ contains $J^1\Mod(U)$.

The shifted upper half plane defined by $\im (z)>1$ defines an open disk $\mathring D$ in $e$. Its  closure $D$ in $X$ is a closed disk 
with  $D\cap P=\{p_0\}\subset \p D$. Contraction of $D$ in $X$ gives a copy of the pair $(X,P)$ and from this it is easily seen that any  orientation preserving self-homeomorphism of $X$ which preserves both $P$ and $p_o$ is isotopic relative the pair $(P,p_0)$ to one which is the identity on  all of  $D$.
If  we identify $\mathring D$ with its image  in $U^*$, then its complement admits $U_1^*$ as a deformation retract and  $\p D$ , when traversed positively,  represents  $\zeta\in \pi_{U_1}$.  This shows that every element of $\Mod(U)_{p_0}$ is realized by an automorphism of $\pi_{U_1}$ which fixes $\zeta$.
\end{proof}

 In general the image of $\Mod(U)$ in  $\aut (\pi_{U^*})$ is a proper subgroup, as it will also preserve the collection of conjugacy classes of the 
arcs that connect two distinct points of $P$ (such an arc gives a loop in $U^*$).  As mentioned in the introduction, the following corollary is due to
Bianchi-Miller-Wilson in the case $P$ is a singleton.

\begin{corollary}\label{cor:johnsonkernel}
The Johnson group $J^k\Mod(U)$ acts trivially  on $H^k(\conf_N(U))$ for all $k$.
\end{corollary}
\begin{proof}
For $k=0$ there is nothing to show, and so assume $k\ge 1$. By Lemma \ref{lemma:stabilizer_p_0}, the group  $J^k\Mod(U)$ {lies in} the $\Mod(U)$-stabilizer of $p_0$ and  hence is realized by by an endomorphism of $\pi_{U_1}$ which fixes $\zeta$. Since $H^k(\conf_N(U))\cong 
H^{cl}_{2n-k}(\conf_N(U))$ by duality,  it remains to apply Corollary  \ref{cor:zeta-generation}.
\end{proof}

\subsection{Configuration spaces of a closed surface} We derive from the preceding a similar statement about the top cohomology of the configuration spaces of the  closed surface $X$. 
Assume $P=\{*\}$, so that  $U=X\ssm \{*\}$ and $\pi_U=\pi_1(X,*)$. The Birman exact sequence can be understood as saying that $\pi_U$ appears as a point pushing normal subgroup of $\Mod(U)=\Mod(X, \{*\})$ with quotient $\Mod(X)$.  The  action of $\Mod(U)$  on $\pi_U$ is faithful and makes $\pi_U$ act by inner automorphisms so that we have  residual (faithful) outer action  of  $\Mod(X)$ on $\pi_U$.
The Johnson filtration of  $\Mod(X)$ is defined as the image of the Johnson filtration of $\Mod(U)$. So  a mapping class of $X$ belongs to $J^k\Mod(X)$ if and only if  it can be lifted to an  element of  $J_k\Mod(U)$ or equivalently, if  the associated outer automorphism of $\pi_U|_k$ is trivial.

\begin{corollary}\label{cor:closedsurface}
Assume $N\not=\emptyset$. Then the  space $\conf_N(X))$ has the homotopy type of a finite complex of dimension $n+1$.

The group  $J^k\Mod(X)$ acts trivially  on $H^k(\conf_N(X)$ for all $k$, but if $k=n+1$, then even  
$J^{n-1}\Mod(X)$ acts trivially.
\end{corollary}
\begin{proof} 
In order to prove the first assertion, we choose $o\in N$. Then  the 
map $\pi: \conf_N(X)\to X$,  $f\mapsto f(o)$  is a  fiber bundle  whose fiber over $x\in X$ is $\conf_{N\ssm\{o\}}(X\ssm \{x\})$. 
The latter has the homotopy type of a finite complex of dimension $n-1$. The base is a surface  and hence  $\conf_N(X)$ has the homotopy type of a finite complex of dimension $n+1$.  

To prove the second assertion, recall  that we endowed  $X^N=U^{*N}$ with a cell decomposition  for which $\conf_N(U)$ is an open union of cells of dimension $\le n$ with the union of $n$-cells in $\conf_N(U)$ being $\conf_N(U_1)$. This   cell decomposition has the property that $\conf_N(X)$ is also an open union of cells: the cells 
we need to add have exactly one factor equal to $\{*\}$. Thus $\conf_N(X)\ssm \conf_N(U)$ is the disjoint union of the configuration spaces  $\conf_{N\ssm \{i\}}(U)$, $i\in N$. Then $\Cscr_N(X)_\pt$ is  the mapping cone of the boundary map
\[
\p': \Cscr_N(U)[-1]_\pt\to \oplus_{i\in N} \Cscr_{N\ssm \{i\}}(U)_\pt.
\]

This mapping cone regards the above  as a double complex: we put the source in row $1$ (and  multiply its  boundary operator by $-1$) and put the target in row $0$ (this makes the target a subcomplex). The  homology in degree $l$ of the  associated single complex is  then   $H^{cl}_l(\conf_N(X))\cong H^{2n-l}(\conf_N(X))$.
An argument similar to the proof of Corollary \ref{cor:johnsonkernel} shows  that $J^{2n-l}\Mod(U)$ acts trivially on  the  $l$th summand of this complex and we then conclude as before. The case $l=n-1$ case is special, because in that degree the double complex is reduced to 
$\oplus_{i\in N} \Cscr_{N\ssm \{i\}}(U)_{n-1}$.
\end{proof}

\subsection{An investigation of top level behavior}\label{subsection:toplevel}  That moment has certainly arrived. Returning to the punctured surface $U$, let us write
\[
S_\pt (U):=\oplus_{n=0}^\infty \Ical_{U}^{n}/\Ical_{U}^{n+1}
\]
for the graded algebra associated to the filtration defined by the augmentation ideal. This algebra is generated in degree one 
(which is just $H_1(U^*)$)
The similarly defined $S_\pt (U_1)$ has  a degree one part canonically isomorphic to $S_1 (U)$ (for $H_1(U_1^*)\to H_1(U^*)$ is an isomorphism)
and is in fact the tensor algebra on $S_1 (U)$.

The natural algebra homomorphism $\Lambda_{U_1}\to \Lambda_U$ is a surjection whose kernel is the 2-sided ideal generated by the image of  $\zeta-1$. A similar statement holds
for the  completions with respect to their augmentation ideals:  $\hat \Lambda_{U_1}\to \hat \Lambda_{U}$ is a surjective homomorphism  
whose kernel is the closure of the  2-sided ideal generated by the image of  $\zeta-1$ in $\hat \Lambda_{U_1}$.  

Let $\mu\in S_2(U_1)$ represents the  fundamental class $[X]$ of $X$, i.e.,  the image of $[X]\in H_2(X)$ under the diagonal map  $H_2(X)\to H_1(X)^{\otimes 2}$ composed with   $H_1(X)^{\otimes 2}\to H_1(U^*_1)^{\otimes 2}=S_2(U_1)$. In terms of our generators:
\begin{equation}\label{eqn:muX}
\textstyle \mu:=\sum_{i=1}^g(a_i\otimes a_{-i}-a_{-i}\otimes a_i)\in H_1(X)^{\otimes 2}\subset H_1(X,P)^{\otimes 2}=S_1(U_1^*).
\end{equation}

The following is well-known. We omit the proof as we will prove a stronger assertion in Lemma \ref{lemma:expansion_delta}.

\begin{lemma}\label{lemma:gradepass}
We have  $\zeta\equiv 1+\mu \pmod{\Ical^3_U}$ and hence the algebra homomorphism $\hat \Lambda_{U_1}\to \hat \Lambda_{U}$ gives after passage to their augmentation gradings
a surjective homomorphism of graded algebras
\begin{equation}\label{eqn:graded}
S_\pt (U_1):=\oplus_{n=0}^\infty \Ical_{U_1}^{n}/\Ical_{U_1}^{n+1}\to  \oplus_{n=0}^\infty \Ical_{U}^{n}/\Ical_{U}^{n+1}=: S_\pt (U)
\end{equation}
whose  kernel is the two-sided ideal generated by the element $\mu\in S_2(U_1)$.\hfill $\square$
\end{lemma}

So  the map \eqref{eqn:graded} in degree $n$, 
\[
H_1(U^*)^{\otimes n}\cong S_n(U_1)\to S_n(U),
\]
has as kernel the sum of the images of $S_{n-2}(U_1)=H_1(U^*)^{\otimes (n-2)}$ under the insertions of $\mu$ in the tensor slots of $H_1(U^*)^{\otimes n}$  in positions $(i, i+1)$, $i=1,2,\dots, n-1$; it will be useful to have the notation  
\begin{equation}\label{eqn:mu-tensor}
    \mu_{i,j}:H_1(U^*)^{\otimes(n-2)}\to H_1(U^*)^{\otimes n}
\end{equation}
for the insertion of $\mu$ in the tensor slots $(i,j)$ for any $1\le i<j\le n$.
By our definition, we  have $S_n(U_1)=\Ical_{U_1^*}^{n}/\Ical_{U_1^*}^{n+1}=
F^n(\Ical_{U_1}|_n)$ and  $S_n(U)=\Ical_{U^*}^{n}/\Ical_{U^*}^{n+1}=F^n(\Ical_{U}|_n)$ and so what we just described
is also the kernel of the map $F^n(\Ical_{U_1}|_n)\to F^n(\Ical_{U}|_n)$.  

We compare this  with the map
$F^n\Hcal_n(U_1)\to F^n\Hcal_n(U)$. Theorem \ref{eqn:tensormap1} tells us  that  the action of $\Mod(U)$ on $\Hcal_n(U)$ is through its action on
$\Ical_U|_n$. By definition, $F^n\Hcal_n(U)$ is spanned by the images of the maps \eqref{eqn:tensormap1} with all $n_i$ equal to $1$. In other words, this is the image of the map 
\begin{equation}\label{eqn:Fnrestriction}
S_1(U)^{\otimes n}\cong H_1(U_*,*)^{\otimes n}\cong  H_n((U_*,*)^n) \to  H_n((U^*)^n, D_n(U))=\Hcal_n(U).
\end{equation}
This is just the map  induced by the inclusion of pairs $(U_*,*)^n\subset (U_*^n, D_n(U))$. Note that it  is $\Sfrak_n$-equivariant. In the same way, we find that $S_1(U_1)^{\otimes n}$ maps $\Sfrak_n$-equivariantly onto $F^n\Hcal_n(U_1)$.

\begin{corollary}\label{cor:}
The map $S_n(U_1)=S_1(U_1)^{\otimes n}\to F^n\Hcal_n(U_1)$ is an isomorphism. The kernel of its composite with the  surjection  
$F^n\Hcal_n(U_1)\to F^n\Hcal_n(U)$ contains 
the sum of the images of $\mu_{i,j}$ for all $1\le i<j\le n$. In particular, $S_n(U_1)\to F^n\Hcal_n(U)$ factors through  a surjection $S_n(U)\to F^n\Hcal_n(U)$.
\end{corollary}
\begin{proof}
The map $S_n(U_1)=S_1(U_1)^{\otimes n}\to F^n\Hcal_n(U_1)$ is defined on the cellular level: it assigns to a product of $1$-cells a sum of $n$-cells in $\conf_n(U_1)$. As any $n$-cell in $\conf_n(U_1)$ appears precisely  once,  the first assertion follows.

For the proof of the second assertion, we  use  the presentation \eqref{eqn:degree_n_pres}: if we take there   $N=[n]$ and $I=\{i,j\}$, then  it tells us  that  the image of the  map 
\[
\psi\in \Hcal_{[n]\ssm I}(U_1)\mapsto \tri^{I}_{U_1}(\zeta)\times\psi\in  \Hcal_n(U_1)
\]
 is in the kernel of  the map $\Hcal_n(U_1)\to \Hcal_n(U)$. If we restrict this map to   $F^{n-2}\Hcal_{[n]\ssm I}(U_1)\cong S_{n-2}(U_1)$, then 
by Lemma \ref{lemma:gradepass}, the right had side is obtained by the insertion of $\mu$ in the tensor slots of $S_n(U_1)$  in positions  $(i,j)$.  In  particular  the resulting map $\mu_{i,j}:S_{n-2}(U_1)\to S_n(U_1)$  is in the kernel of $S_n(U_1)\to F^n\Hcal_n(U)$.
\end{proof}

If the genus of $X$ is at least $2$ and $n\ge 3$, then the  $(1,n)$-insertion of $S_{n-2}(U)$ is certainly not contained in the sum of its $(i, i+1)$ -insertions and so the induced map  $S_n(U)\to F^n\Hcal_n(U)$ is then not injective.

\begin{remark}\label{rem:}
In subsequent work of the second author it will be shown  that the kernel contains the $\ZZ\Sfrak_n$-span of 
$\mu \otimes S_1(U)^{\otimes (n-2)}$ as a subgroup of finite index.
\end{remark}

 \subsection{A study of the kernel $\Kcal_n(U,U_1)$} 
Theorem \ref{thm:kernelmultiplecells} described the kernel $\Kcal_n(U, U_1)=\Ker\big(\Hcal_n(U_1)\to \Hcal_n(U)\big)$,
of which we have so far only studied the intersection with $F^n\Hcal_n(U_1)$, as the span of the $\Scal_n$-orbit of $\sum_{k=1}^n\tri^k_{U_1}(\zeta)\times \Hcal_{n-k}(U_1)$. It is then desirable to compute some of the elements $\tri^k_{U_1}(\zeta)$ for small $k$. We here recall that 
\begin{multline}\label{eqn:tri_exp}
\textstyle \tri_{{U_1},t}(\zeta)=\sum_{n=0}^\infty \tri_{U_1}^n(\zeta)t^n=\\
\tri_{{U_1},t}(\alpha_1)\tri_{{U_1},t}(\alpha_{-1})\tri_{{U_1},t}(\alpha_{1})^{-1}\tri_{{U_1},t}(\alpha_{-1})^{-1}\tri_{{U_1},t}(\alpha_2)\cdots 
\tri_{{U_1},t}(\alpha_{-g})^{-1} \in \widehat \Hcal_\pt({U_1}),
\end{multline}
where we used that $\tri_{{U_1},t}$ is an algebra homomorphism.
So this is a formal power series  in $t$ with constant term $1$ and  whose  $n$th  coefficient $\tri_{U_1}^n(\zeta)$ is a  noncommutative polynomial in  the $\tri_{U_1}^j(\alpha_i)$. That polynomial  is homogeneous  of degree $n$, if we stipulate that $\tri_{U_1}^j(\alpha_i)$ has degree $j$.  The expansion of $\tri_{{U_1},t}(\zeta)$ can of course be computed up to arbitrary order.

\begin{lemma}\label{lemma:expansion_delta}
We have that $\tri_{{U_1}}^n(\zeta)\in F^2\Hcal_n({U_1})\ssm F^3\Hcal_n({U_1})$ for all $n>0$. Furthermore $\tri^1_{U_1}(\zeta)=0$, $\tri^2_{U_1}(\zeta)=\mu$ and
\begin{equation*}
\textstyle \tri_{U_1}^3(\zeta)= \sum_{i=\pm 1}^{\pm g} \sign (i)[\tri^2(\alpha_i), a_{-i}] -\sum_{i=1}^g [a_i,a_{-i}](a_i+a_{-i}) 
\end{equation*}
where $a_i$ is identified with $\tri_{U_1}^1(\alpha_i)$, and so
\[
\Kcal_i(U,U_1)=
\begin{cases}
0 &\text{if $i=1$;}\\
\ZZ\mu  &\text{if $i=2$;}\\ 
\ZZ \tri^3_{U_1}(\zeta)+\mu\times\Hcal_1(U_1)+\Hcal_1(U_1)\times \mu +\mu_{1,3}(\Hcal_1(U_1)) &\text{if $i=3$.}\\ 
\end{cases}
\]
Here, the term $\mu_{1,3}(\Hcal_1(U_1))$ is the image of
$$
 \Hcal_1(U_1)\cong S_1(U_1)\xrightarrow{\mu_{1,3}} S_1(U_1)^{\otimes 3}\xrightarrow{\cong} F^3\Hcal_3(U_1).
$$
\end{lemma}
\begin{proof}
We first check the expansions. We find  that $\tri_{U_1}^1((\alpha_i, \alpha_{-i}))= a_i+a_{-i}+ (-a_i) + (-a_{-i})=0$ and hence that $\tri_{U_1}^1(\zeta)=0$.
 Similarly, we get
 \begin{multline*}
\tri_{U_1}^2((\alpha_i, \alpha_{-i}))=\tri_{U_1}^2(\alpha_i) + a_i\otimes (a_{-i}+ (-a_i)+(-a_{-i}))+ \\
+ \tri_{U_1}^2(\alpha_{-i}) +  a_{-i}\otimes ((-a_i)+(-a_{-i}))+\tri_{U_1}^2(\alpha_{i}^{-1})+ (-a_{i})\otimes (-a_{-i}) + \tri_{U_1}^2(\alpha_{-i}^{-1})=\\=
a_i\otimes a_{-i}-a_{-i}\otimes a_i, 
\end{multline*}
which implies that $\tri_{U_1}^2(\zeta)=\mu$. The expression for  $\tri_{U_1}^3(\zeta)$ is obtained likewise.

For the first assertion, we apply twice the composition formula modulo $F^2\Hcal_n(U_1)$ to obtain  
$$
\textstyle \tri^n_{U_1}(\zeta)\equiv \sum_{i=\pm 1}^{\pm g} \tri^n_{U_1}(\alpha_i)+\tri^n_{U_1}(\alpha_i^{-1})\equiv \sum_{i=\pm 1}^{\pm g} \tri^n_{U_1}(\alpha_i\alpha_i^{-1})=\sum_{i=\pm 1}^{\pm g} \tri^n_{U_1}(1)= 0,
$$
and so $\tri^n_{U_1}(\zeta) \in F^2\Hcal_n(U_1)$. To show that $\tri^n_{U_1}(\zeta)$ does not lie $F^3\Hcal_n(U_1)$, apply the iterated degeneracy $\partial^2_n:\Hcal_n(U_1)\to \Hcal_2(U_1)$ from Lemma \ref{lem:iterateddegeneracy}: this map annihilates $F^3\Hcal_n(U_1)$ but maps $\tri^n_{U_1}$ to $\pm \tri^2_{U_1}(\zeta)\neq 0$.
\end{proof}

\begin{remark}\label{rem:}
Lemma \ref{lemma:expansion_delta} shows that  $\tri^3_{U_1}(\zeta)$  is not divisible by $\mu$, which implies that $\tri^n_{U_1}(\zeta)$ is also not divisible by $\mu$ for $n\ge 3$.
\end{remark}

If $I$ is a $2$-element set, then $\Kcal_I(U, U_1)$ is a copy of $\Kcal_2(U, U_1)=\ZZ\mu$. Its generator is defined up to sign (this requires that we order $I$) and  we let $\mu_I$ stand for one of these generators.

\subsection{Towards the  kernel of the mapping class group representation} 
We  denote  the image of $\tri^n_U:  \Lambda_{U}|_n\to \Hcal_n(U)$ by $\Lambda^\cfg_{U,n}$:
\begin{equation}\label{eqn:flatquotient}
\Lambda_{U}|_n\twoheadrightarrow  \Lambda^\cfg_{U,n}\hookrightarrow \Hcal_n(U).
\end{equation}
It is of course  obtained  as the quotient of $\Lambda_{U_1}|_n$ by the  preimage of $\Kcal_n(U, U_1)$ in  $\Lambda_{U}|_n$.  For $n>0$ this is also the image of $\Ical_U|_n$ and so we may then write $\Ical^\cfg_{U,n}$ instead.
 Lemma \ref{lemma:truncation} shows that  for  $n>0$
the truncation $\Ical_{U}|_{n+1}\twoheadrightarrow \Ical_{U}|_{n}$ induces a surjection $\Lambda^\cfg_{U,n+1}\twoheadrightarrow \Lambda^\cfg_{U,n}$ that makes the  diagram  below commute
\begin {equation}\label{eqn:truncflat}
    \begin{tikzcd}
    \Ical_{U}|_{n+1}\dar[two heads]\rar[two heads] &\Ical^\cfg_{U,n+1}\dar[dashed,two heads]\rar[hook] & \Hcal_{n+1}(U)\dar[two heads,"trunc "]\\
    \Ical_{U}|_n\rar[two heads] &   \Ical^\cfg_{U,n}\rar[hook]  &\Hcal_n(U)
\end{tikzcd}
\end{equation}
It is clear that $\Ical^\cfg_{U,n+1}\to  \Ical^\cfg_{U,n}$ will factor through a surjection  
$\Ical^\cfg_{U,n+1}|_n \to \Ical^\cfg_{U,n}$, but we do not claim that the latter map is an isomorphism (in fact it is usually not,  as will be shown in
the sequel in preparation by A.S.).

The natural  $\Mod(U)$-equivariant surjection of Theorem \ref{thm:generation} clearly factors through the $\Mod(U)$-equivariant map
\begin{equation}\label{eqn:tensormap3}
\oplus_r\oplus_{(\Nbold_1,\dots , \Nbold_r)} \;  \Lambda^\cfg_{U,n_1}\otimes\Lambda^\cfg_{U,n_2}\otimes\cdots  \xrightarrow{\tri^{\Nbold}} \Hcal_n(U),
\end{equation}
where we are summing over the decompositions of $[n]$ into nonempty parts with a specified total order on each part.
In view of the surjectivity of $\Lambda^\cfg_{U,n+1}\to\Lambda^\cfg_{U,n}$ (by \eqref{eqn:truncflat}), we find:

\begin{corollary}\label{cor:}
The action of $\Mod(U)$ on $\Hcal_n(U)=H^n(\conf_n(U))$ has the same kernel as its action on $\Lambda^\cfg_{U,n}$. \hfill $\square$
\end{corollary}

We denote  this common kernel by  $J_\cfg^n\Mod(U)\subset \Mod(U)$. 
It follows from the diagram \eqref{eqn:truncflat} that  the $J_\cfg^n\Mod(U)$ form  a descending series of normal subgroups of $\Mod(U)$. It is also clear that $J_\cfg^n\Mod(U)\supset J^n\Mod(U)$. This inclusion could be strict and we shall see in Section \ref{subsect:comparison}, that  this is indeed usually the case. 

\begin{remark}
Here is a description of the projective limit  $\hat\Lambda^\cfg_{U}:=\varprojlim_n\Lambda^\cfg_{U,n}$ without reference to a configuration space. It will  not be used in what follows. The  characterization that we gave in  Subsection \ref{subsect:R} of the monoidal functor $\uHcal(\RR)=\uHcal$ shows  that  there is an algebra homomorphism given by 
 \[
\Scal: \widehat \Hcal_\pt\to \QQ[[t]], \quad \tri^{\Nbold}t^n\mapsto t^n/n!
\]
It  has the property that it  takes $\sum_{n\ge 0} \tri^n(\g_o)t^n=\sum_{n\ge 0} \tri^n t^n$  to the formal power series expansion of $e^t$.
Since the map $\tri_{\RR,t}: \hat\Lambda_\RR\to \widehat\Hcal_\pt$ is an algebra homomorphism as well,  its composite with 
$\Scal$ is the exponential map. 

This  generalizes to the case where $\RR$ is replaced by disjoint union of a finite number of copies of 
$\RR$. For this generalization, we then better regard the  target of $\Scal$ as  the divided power series ring (which here is its image). 
We then find an algebra homomorphism
\begin{equation}\label{eqn:divpowers}
\textstyle \Scal_{U_1}: \widehat \Hcal_\pt(U_1)\to \prod_{n=0}^\infty S_n(U_1)t^n/{n!}
\end{equation}
which takes $\tri^{\Nbold}(\alpha_i)$ to $a_i^{\otimes n}t^n/{n!}$. We do likewise for $b_j$. The map  $\Scal_{U_1}$ is then surjective  and   
its  precomposite  with $\tri_{U_1, t}: \hat\Lambda_{U_1}\to \widehat \Hcal_\pt(U_1)$ 
takes   $\alpha_i$ to $e^{ta_i}$ and $\beta_j$ to $e^{tb_j}$. 
This composite is in fact an embedding, 
because the $e^{ta_1} , \dots, e^{tb_l}$ are formally independent  (it is even an isomorphism when tensored with $\QQ$). 

If we want this to descend to $U$ we must do a pushout. The preceding shows how: let  $S^\cfg_\pt\lla  U\rra$ be the (graded) quotient of the right hand side of \eqref{eqn:divpowers} by the smallest homogeneous 2-sided ideal which contains 
\begin{multline*}
(e^{ta_1}, e^{-ta_{-1}})\cdots (e^{ta_g}, e^{-ta_{-g}})-1=\\
\textstyle =\mu_Xt^2+  \big( \sum_{i=\pm 1}^{\pm g} \sign (i)[\frac{1}{2}a_i^2, a_{-i}] -\sum_{i=1}^g [a_i,a_{-i}](a_i+a_{-i}) \big)t^3+\cdots 
\end{multline*}
and  is $\Scal_n$-invariant in degree $n$. This ensures that there is an induced algebra  homomorphism 
$\widehat \Hcal_\pt(U)\to S^\cfg_\pt\lla  U\rra$. Its precomposite with $\tri_{U, t}: \hat\Lambda_{U}\to \widehat \Hcal_\pt(U_1)$ is the natural map
and factors of course through the image of $\hat\Lambda_{U}$ in $ \widehat \Hcal_\pt(U_1)$. 
This image  is equal to $\hat\Lambda^\cfg_{U}=\varprojlim_n\Lambda^\cfg_{U,n}$ and embeds in $S^\cfg_\pt\lla  U\rra$. 
\end{remark}

\subsection{Many punctures in terms of one}\label{sec:manypunctures}
We conclude this section with a proof of Theorem \ref{thm:manypuncturesalldegrees}. The key observation is that the word $\zeta$ does not involve the $\beta_j$-loops. Then the differential $\partial$ of the complex $\underline{\Cscr}(U)_\pt$, which interacts with the subcomplex $\underline{\Cscr}(U_1)_\pt$ only via the map $\times_{\zeta}$, also does not involve the $\beta_j$-cells, so that the $\beta_j$-part splits off in the form of the decomposition of $\FB$-chain complexes 
\begin{equation}\label{eq:decompositionmanypunctures}
    (\underline{\Cscr}(U)_\pt, \partial)\cong (\underline{\Cscr}(X\ssm \{p_0\})_\pt, \partial)\otimes (\uHcal^{\otimes l}, 0).
\end{equation}
Here the second factor embeds in $(\underline{\Cscr}(U)_\pt, \partial)$ using the maps $$\beta_j:(\PP^1(\RR),\infty)\to (U_1^*, *)$$ for $j=1,\ldots, l$.
Taking homology, tells us that the $\FB$-functor  $N\mapsto H^{cl}_\pt(\conf_N (U))$ is the $\FB$-tensor product of the functors $N\mapsto H^{cl}_\pt( \conf_N (X\ssm \{p_0\}))$ and $\uHcal^{\otimes l}$. Theorem \ref{thm:manypuncturesalldegrees} is then a spelling out of this $\FB$-product and an application of  Poincar\'e--Lefschetz duality.

This isomorphism is  natural to some extent. In order to  exhibit this, we choose a closed disk $D\subset X$ which contains the  union of the paths $\beta_1,\ldots, \beta_l$ such that this union meets 
$\p D$ in $\{p_0\}$ only. Then viewing the subcomplex $(\underline{\Cscr}(X\ssm \{p_0\})_\pt, \partial)$ as $(\underline{\Cscr}(X\ssm D)_\pt, \partial)$ we get an action by $\Mod(X\ssm D)$ on the former complex. Extending by the trivial action on the second factor in the decomposition \eqref{eq:decompositionmanypunctures} and acting naturally on the left, makes the isomorphism on homologies $\Mod(X\ssm D)$-equivariant.

\section{Comparison with the Johnson filtration}\label{subsect:comparison}
We show that in general   the Johnson filtration strictly refines the modified version defined above. For a class of  examples  we take 
 $P$  to be a singleton $\{*\}$ (so that $U^*=X$)  and  assume that the genus $g$ of $X$  is $\ge 2$. 

We simplify notation by putting $S_n:=S_n(X)$, $\pi:=\pi_1(X,*)$ and denote the associated graded Lie algebra by $L_\pt:=\oplus_{n\ge 1}\pi^{(n)}/ \pi^{(n+1)}$. We recall that this Lie algebra is generated by  $H_1(X)$ and has  $\sum_{i=1}^g [a_i, a_{-i}]=0$ as a defining relation. Its universal enveloping algebra  $S_\pt$ is the tensor algebra on $H_1(X)$ modulo the 2-sided ideal generated by $\mu=\sum_{i=1}^g (a_i\otimes a_{-i}-a_{-i}\otimes a_i)$.
We write $\Ical\subset \ZZ\pi$ for the augmentation ideal.

The group $\Mod(X,*)$ acts faithfully on $\pi$ and its image contains the group of inner of automorphisms of  $\pi$.
Since $g\ge 2$,  $\pi$ has trivial center  and so this defines an embedding  $\iota: \pi\to \Mod(X,*)$. It is realized  geometrically by `point pushing' the base point $*$ along a loop. 

Since $J^n\Mod(X,*)$ is the kernel of the action of $\Mod(X,*)$ on $\pi_n=\pi/ \pi^{(n+1)}$, it follows that if $\alpha\in \pi$, then 
$\iota(\alpha)\in J^n\Mod(X,*)$ if and only $(\alpha, \beta)\in \pi^{(n+1)}$ for all $\beta\in \pi$. The latter is implied by (but is in fact equivalent with) 
$\alpha\in \pi^{(n)}$. In particular, we get a homomorphism of $\ZZ$-modules  
\[
\gr^n(\iota):  L_n\to J^n\Mod(X,*)/J^{n+1}\Mod(X,*)
\]

If $\varphi\in \Mod(X,*)$, then  $\varphi\in J^n\Mod(X,*)$ if and only if $\varphi_*-1$ maps $\ZZ\pi$ to $\Ical^n$. It follows that the assignment  $\varphi\mapsto \varphi_*-1$   defines an embedding of (abelian) groups
\[
J^n\Mod(X,*)/J^{n+1}\Mod(X,*)\hookrightarrow \Hom(S_1, S_{1+n}).
\]
This is  the $n$th Johnson homomorphism. If we precompose this map with $\gr^n(\iota)$, we obtain 
\[
L_n\to \Hom(S_1, S_{1+n}), \quad \alpha\mapsto (b\mapsto \alpha b\alpha^{-1}-b)
\]
which is induced by $\alpha\mapsto \alpha-1$. Equivalently, the map is given by
\[
\ad^1_{1+n}: L_n\to \Hom(S_1, S_{1+n}), \quad \ad^1_{1+n}(a)(x)= [a,x]
\]
So if we succeed in finding an $a\in L_n$ with the property that $\ad^1_{1+n}(a)$ is nonzero, but takes its values in the kernel of 
$S_{1+n}\to S^\cfg_{1+n}$, then we have obtained a nonzero element of $J_\cfg^{n-1}\Mod(X,*)\ssm J^{n-1}\Mod(X,*)$.
The map $S_\pt\to  S^\cfg_\pt$ is an algebra homomorphism and hence the  last  condition is certainly fulfilled if we take  
$a$ itself in the kernel of  $L_n\to S^\cfg_n$. For example, if $n\ge 2$, then we may take 
\[
\textstyle a=\sum_{i=1}^g  \Big([a_i[c_1,[c_2,[\cdots , [c_{n-2},a_{-i}]\cdots]\!]\!] -[a_{-i},[c_1,[c_2,[\cdots , [c_{n-2},a_{i}]\cdots\!]\!]\!] \Big),
\]
where $c_1, \dots , c_{n-2}$ are in $S_1$. It then remains to find such $c_\nu$ such that $\ad^1_{1+n}(a)\not=0$. 

In the simplest possible nontrivial case, 
namely when $n=3$, a straighforward computation shows that then $a=0$. We therefore take $n=4$. Let us write out one of its terms:
\begin{multline*}
[a_i[c_1,[c_2,a_{-i}]]]=[a_i, c_1(c_2a_{-i}-a_{-i}c_2)-(c_2a_{-i}-a_{-i}c_2)c_1]=\\
=a_i\big(c_1(c_2a_{-i}-a_{-i}c_2)-(c_2a_{-i}-a_{-i}c_2)c_1\big)-\big(c_1(c_2a_{-i}-a_{-i}c_2)-(c_2a_{-i}-a_{-i}c_2)c_1\big)a_i
\end{multline*}
It is convenient to collect the terms of the right hand side according to the two positions in which the basis elements  $a_j$ appear:
\begin{small}
\begin{center}
\begin{tabular}{|c|c|}
\hline 
position & terms \\
\hline 
(1,2) & $a_ia_{-i}c_2c_1$\\
(1,3) &  $-a_ic_1a_{-i}c_2-a_ic_2a_{-i}c_1$\\
(1,4) &  $a_ic_1c_2a_{-i}-a_{-i}c_2c_1a_i$\\
(2,3) & 0 \\
(2,4) & $c_1a_{-i}c_2a_i+c_2a_{-i}c_1a_i$\\
(3,4) & $-c_1c_2 a_{-i}a_i$\\
\hline
\end{tabular}
\end{center}
\end{small}
In order to get $a$, we must  sum over $i=1, \dots , g$ and subtract from this the same expression with 
$a_i$ and $a_{-i}$ exchanged.  Since $\sum_i a_ia_{-i}-a_{-i}a_{i}=0$,  we only get contributions from
the `nonadjacent positions' and  thus find that 
\[
a= (-\mu_{1,3}+\mu_{1,4}-\mu_{2,4})(c_1c_2+c_2c_1)
\]
(see the defining identity \eqref{eqn:mu-tensor} for $\mu_{i,j}$). This confirms once more that $a$ maps to zero in $S^\cfg_4$. But it also shows that $\ad^1_5(a)\not=0$ for general $c_1, c_2$: if we put 
$c:=c_1c_2+c_2c_1$, then for $x\in S_1$, we have 
\[
\ad^1_5(a)(x)= (-\mu_{1,3}+\delta_{1,4}-\mu_{2,4})(cx)-(-\mu_{2,4}+\delta_{2,5}-\mu_{3,5})(xc)
\]
which is indeed  nonzero for a general triple $(c_1, c_2,x)$.
For a concrete example, take $c_1=a_1$, $c_2=a_{a_{-1}}$ and $x=a_{2}$, then the coefficient of $a_2a_1a_{-2}a_{-1}a_2$ in $\ad^1_5(a)(x)$ is $-2\neq 0$.

\end{document}